\documentclass[11pt]{amsart}

\usepackage{a4wide, amsmath, amsfonts, amssymb, amsthm, bbm}
\usepackage[hyperfootnotes=false]{hyperref}
\usepackage[shortlabels]{enumitem}

\numberwithin{equation}{section}

\allowdisplaybreaks[2]

\newtheorem{theorem}{Theorem}[section]
\newtheorem{lemma}[theorem]{Lemma}
\newtheorem{corollary}[theorem]{Corollary}
\newtheorem{remark}[theorem]{Remark}
\newtheorem{proposition}[theorem]{Proposition}
\newtheorem{definition}[theorem]{Definition}
\newtheorem{assumption}[theorem]{Assumption}

\renewcommand{\epsilon}{\varepsilon}
\renewcommand{\P}{\mathbb{P}}
\newcommand{\dd}{\,\mathrm{d}}
\renewcommand{\d}{\mathrm{d}}
\newcommand{\Id}{\,\mathrm{d}}
\newcommand{\R}{\mathbb{R}}
\newcommand{\N}{\mathbb{N}}
\newcommand{\E}{\mathbb{E}}
\newcommand{\indicator}[1]{\mathbbm{1}_{#1}}
\newcommand{\norm}[1]{\left\lVert#1\right\rVert}
\newcommand{\testfunctions}[1]{C_c^{\infty}(#1)}

\title[Diffusion-aggregation equations and their mean-field approximations]{Well-posedness of diffusion-aggregation equations with bounded kernels and their mean-field approximations}

\author[Chen]{Li Chen}
\address{Li Chen, University of Mannheim, Germany}
\email{chen@uni-mannheim.de}

\author[Nikolaev]{Paul Nikolaev}
\address{Paul Nikolaev, University of Mannheim, Germany}
\email{pnikolae@mail.uni-mannheim.de}

\author[Pr{\"o}mel]{David J. Pr{\"o}mel}
\address{David J. Pr{\"o}mel, University of Mannheim, Germany}
\email{proemel@uni-mannheim.de}

\date{\today}

\begin{document}

\begin{abstract}
  The well-posedness and regularity properties of diffusion-aggregation equations, emerging from interacting particle systems, are established on the whole space for bounded interaction force kernels by utilizing a compactness convergence argument to treat the non-linearity as well as a Moser iteration. Moreover, we prove a quantitative estimate in probability with arbitrary algebraic rate between the approximative interacting particle systems and the approximative McKean--Vlasov SDEs, which implies propagation of chaos for the interacting particle systems.
\end{abstract}

\maketitle




\section{Introduction}

Diffusion-aggregation equations and their associated interacting particle systems serve as well-suited mathematical models in various areas, such as physics, chemistry, biology, ecology, and social sciences. For instance, they are used to describe the behaviour of chemotaxis~\cite{KellerSegel1970,HillenPainter2009,Horstmann2004}, angiogenesis and swarm movement~\cite{Topaz2006}, flocking~\cite{HaLiu2009}, opinion dynamics~\cite{LORENZ_2007_bounded_confidence_survey, noorazar2020recent}, and cancer invasion~\cite{Domschke2014}. On the microscopic level, these systems are often modeled by interacting \(N\)-particle systems \(\mathbf{X}^N= (X^1,\ldots, X^N)\), given by stochastic differential equations of the form
\begin{equation}\label{eq: intro_particle_system}
  \Id X_t^{i} = -\frac{1}{N} \sum\limits_{j=1}^N k(X_t^{i} -X_t^{j}) \Id t + \sigma \Id B_t^{i}, \quad i=1,\ldots,N , \; \; \mathbf{X}_0^N \sim \overset{N}{\underset{i=1}{\otimes}} \rho_0,
\end{equation}
for $t\geq0$, starting from i.i.d. initial data, defined on a probability space \((\Omega, \mathcal{F},\P)\). On the macroscopic level, the corresponding systems are represented by the evolutions of the probability densities~\(\rho\) of the particles, which satisfy diffusion-aggregation equations. In general, these diffusion-aggregation equations are non-local, non-linear partial differential equations (PDEs). Passing from the microscopic to macroscopic models, involves to study the mean-field limit as \(N\to\infty\), cf. \cite{snitzman_propagation_of_chaos,CarilloChoiPil2014,JabinWangZhenfu2016,JabinEmanuel2014}.
In particular, this consists of showing the convergence of the empirical measures $\mu^N_t$ of the \(N\)-particle systems \(\mathbf{X}^N= (X^1,\ldots, X^N)\) for all \(t \ge 0\), where $\mu^N_t$ is defined as
\begin{equation*}
  \mu^N_t(\omega,A) := \frac{1}{N} \sum\limits_{i=1}^N \delta_{X^{i}_{t}(\omega)}  (A), \quad \quad  \omega \in \Omega
\end{equation*}
for a Borel set \(A\). 
Although mean-field interaction and its related PDEs is a classical topic, it is still a very active research field. Indeed, the case of global Lipschitz continuous interaction force kernels~\(k\) has been understood for many years, cf. \cite{McKean1967,snitzman_propagation_of_chaos,HaurayMischler2014}, e.g., by employing the coupling method, that is, comparing the particle \((X_t^{i},t \ge 0)\) to the solution \((Y_t^{i},t \ge 0)\) of the McKean--Vlasov stochastic differential equations (McKean--Vlasov SDEs)
\begin{align*}
  \begin{cases} \Id Y_t^{i} = -  (k* \mu_t)(Y_t^{i}) \Id t + \sigma \Id B_t^{i},  \quad i=1,\ldots,N , \; \; \mathbf{Y}_0^N=\mathbf{X}_0^N \\
  \mu_t = \mathrm{Law}(Y_t)
  \end{cases}
\end{align*}
for $t\geq0$, and, subsequently, showing the convergence \(\mu^N_t \to \mu_t\) as  \(N\to \infty\) for all \(t \ge 0\) in a suitable topology. The latter convergence is also referred to as ``propagation of chaos''. Consequently, the question regarding the well-posedness of McKean--Vlasov SDEs naturally arises in the context of mean-field theory, cf. \cite{Wang2018,rockner_DDSDE2021,barburoeckner2020,HuangPanpanWang2021,Wang2023}.
In many settings, the law \(\mu\) of the solution \((Y_t^{i},t \ge 0)\) possesses a probability density \(\rho\), which satisfies an associated Fokker--Planck equation. Therefore, one has access to PDE theory allowing to deal with the well-posedness of McKean--Vlasov SDEs. 


Motivated by various models arising especially in physics, which require bounded measurable or even singular interaction force kernels, an enormous amount of work has been dedicated to treat such irregular interaction force kernels. Initially, approaches to treat such irregular kernels were often based on compactness methods in combination with the martingale problems associated to the McKean--Vlasov SDEs, see e.g. \cite{Oeschlager1984,Hirofumi1987,Gartner1988,FournierJourdain2017,godinho_Keller_Segel2015,LiLeiLiu2019,LiLeiLiu2019}. More recently, even singular kernels, like the Coulomb potential \(x/|x|^s\) for \( s \ge 0 \), were investigated in the non-random setting~\cite{Serfaty_2020,Nguyen2022} (\(\sigma = 0)\) as well as in a random setting~\cite{JabinWang2018,BreschDidierJabinWangZhenfu2019,bresch2020,RosenzweigSerfaty2023} (\(\sigma > 0 \)). The aforementioned references introduced a novel method called the modulated free energy approach, which provides a practical quantity to obtain a priori estimates. For the Coulomb potential, this quantity even metrize the weak convergence of the empirical measures~\cite{RosenzweigSerfaty2023}. A drawback of the modulated free energy approach is that it requires the existence of an entropy solution on the particle level (microscopic level), see~\cite[Proposition~4.2]{BreschDidierJabinWangZhenfu2019}, which is non-trivial outside a setting on the torus. Further results on propagation of chaos were proven  for general \(L^p\)-interaction force kernels~\(k\) for first and second order systems on the torus~\cite{bresch2023} and on the whole space~\(\R^d\) \cite{hao2022,han2023,Lacker2023}. For instance, \cite{Lacker2023} provides optimal bounds on the relative entropy of order \(\mathcal{O}(k^2/N^2)\) by exploiting the BBGKY-hierarchy combined with delicate estimates on the error of iterations.

An influential approach allowing to deal with the {V}lasov--{P}oisson system, which is a second order system with a singular interaction force kernel~\(k\), was introduced by D. Lazarovici and P. Pickl~\cite{lazarovici2017mean}. For the {V}lasov--{P}oisson system, the underlying particle system~\eqref{eq: intro_particle_system} is a priori not well-posed. Therefore, a regularization \( k^\epsilon\) of the kernel~\(k\) is required, where \(k^\epsilon\) is a smooth approximation of the  interaction force kernel~\(k\) such that the system~\eqref{eq: intro_particle_system} is well-posed. The aforementioned approach is widely used, for instance, for the Keller--Segel equation~\cite{HuangHiuLius2019,LiuYang2019,Fetecau_2019}, the Cucker--Smale model with singular communication~\cite{HaKimPickl2019} and the Vlasov--Poisson--Fokker--Planck equation~\cite{CarrilloChoiSalem2019,HuangLiuPickl2020,ChenLiPickl2020}. An advantage of it is that well-posedness of the underlying particle system is not required since one works directly with the regularized/approximative particle system using the kernel \(k^\epsilon\). In particular, if the system has a non-regular drift, as e.g. the Keller--Segel system~\cite[Proposition~4]{FournierJourdain2017}, the underlying particle system could collapse. Moreover, the approach of D. Lazarovici and P. Pickl allows to show the propagation of chaos of the regularized particle systems to the regularized mean-field equation. That means, it acts like an intermediate result. On the one hand, the remaining limit of the regularized mean-field equation to the mean-field equation is reduced to a convergence analysis on the PDE level. On the other hand, the convergence of the regularized particle system to the non-regularized particles system only requires a stability analysis on the SDE level, which still is, at least in general, a challenging task.

In the present article we establish the approach of D. Lazarovici and P. Pickl~\cite{lazarovici2017mean} in a general setting allowing for interacting particle systems and diffusion-aggregation equations with bounded interaction force kernels which can be approximated in a suitable manner by smooth kernels. One main objective is to provide a transparent road map how to utilize this approach. To that end, we give a brief summary of the approach and explain its core concepts. 

While we present all results in a one-dimensional setting to avoid cumbersome notation, we would like to remark that all results can be extended with minor modifications to a multi-dimensional setting.

The first contribution is the well-posedness of the diffusion-aggregation equation, see~\eqref{eq: aggregation_diffusion_pde} below, which is derived from the interacting particle system~\eqref{eq: intro_particle_system}, for bounded interaction force kernels \(k\). The main challenge lies in the non-linearity in the transport term, which is treated by a strong-weak convergence argument provided by Aubin lemma. The presented well-posedness result expands previous existence results regarding similar PDEs, for instance, regarding bounded confidence models~\cite{Lorenz2007,ChazelleJiu2017} used in social science.

The second contribution is to provide \(L^p\)- and \(L^\infty\)-estimates for the solution \(\rho\) through a Moser iteration. Following~\cite{lazarovici2017mean}, we introduce a uniform local Lipschitz assumption, see Assumption~\ref{ass: loc_lip_bound} below. For instance, we verify that models for the opinion formation of interacting agents, such as the Hegselmann--Krause model~\cite{hegselmann2002}, satisfy this uniform local Lipschitz assumption. As a rule of thumb, Assumption~\ref{ass: loc_lip_bound} is fulfilled by interaction force kernels with jump/singularity having the same order as the space dimension, which in the present case is one.

As third contribution, we establish propagation of chaos in probability supposing the local Lipschitz assumption for the bounded interaction force kernel~\(k\). This is achieved by proving a suitable law of large numbers,  demonstrating the convergence of the regularized particle system to the regularized mean-field system in a suitable topology and, subsequently, proving the convergence of the regularized probability density \(\rho^\epsilon\) to the  probability density \(\rho\) as \(\epsilon \to 0\).

\medskip

\noindent \textbf{Organization of the paper:} In Section~\ref{sec: preliminaries} we introduce the notation, the interacting particle systems and their associated diffusion-aggregation equations. Moreover, we present a brief outline of the used method, building on the work of D. Lazarovici and P. Pickl~\cite{lazarovici2017mean}. The well-posedness and regularity properties of the diffusion-aggregation equations are established in Section~\ref{sec: existence_of_pdes}. In Section~\ref{sec: local_lipschitz_bound} we discuss the local Lipschitz assumption on the approximative interaction force kernels and provide various examples. Section~\ref{sec: law_of_large_numbers} contains the law of large numbers and the propagation of chaos in probability is provided in Section~\ref{sec: convergence_in_probability}.
 
\section{Setting and method}\label{sec: preliminaries}

In this section we introduce the basic setting, that is, the necessary notation, the interacting particle systems as well as their associated PDEs, and outline the general method implemented in the present paper, following~\cite{lazarovici2017mean}.

\subsection{Basic definitions and function spaces}

In this subsection we collect the basic definitions and introduce the required function spaces.

For a vector \(x=(x_1,\ldots,x_N) \in \R^N\), we write $|x|$ for the standard Euclidean norm and \(|x|_\infty = \sup_{1 \le i \le N} |x_i|\) for the \(l^\infty\)-Euclidean norm. Throughout the entire paper, we use the generic constant \(C\) for inequalities, which may change from line to line. For two functions $g$ and $f$, we write $f\sim g$ if they are proportional.

For \( 1 \le p \le \infty\) we denote by \(L^p(\R)\) the space of measurable functions whose \(p\)-th power is Lebesgue integrable (with the standard modification for \(p = \infty\)) equipped with the norm \(\norm{\cdot}_{L^p(\R)}\), by \(L^1(\R, |x|\Id x)\) the space of all measurable functions \(f\) such that \(\int_{\R}|f(x)||x|\Id x <\infty\), by \(\testfunctions{\R}\) the space of all infinitely differentiable functions with compact support on \(\R\), and by \(\mathcal{S}(\R)\) the space of all Schwartz functions, see \cite[Chapter~6]{YoshidaKosaku1995FA} for more details.

Let \((Z,\norm{\cdot}_Z)\) be a Banach space. The space \(L^p([0,T];Z)\) consists of all strongly measurable functions \(u\colon [0,T] \to Z\) such that
\begin{equation*}
  \norm{u}_{L^p([0,T];Z)}:= \left( \int\limits_0^T \norm{u(t)}_Z^p \Id t \right)^{\frac{1}{p}} < \infty
  ,\quad \text{for } 1 \le p < \infty,
\end{equation*}
and
\begin{equation*}
  \norm{u}_{L^\infty([0,T];Z)}
  := \operatorname*{ess\,sup}_{ t\in [0,T]} \norm{u(t)}_Z < \infty, \quad \text{for }p=\infty.
\end{equation*}
The Banach space \(C([0,T];Z)\) consists of all continuous functions \(u \colon [0,T] \to Z\) and is equipped with the norm
\begin{equation*}
  \max\limits_{t \in [0,T]} \norm{u(t)}_Z < \infty.
\end{equation*}

For sufficiently smooth functions \(u \colon [0,T] \times \R \to \R \) we denote the \(n\)-th derivative with respect to \(x\) by \(\frac{\d^n}{\d x^n} u(t,x)\), where we also write \(u_x\) for \(\frac{\d}{\d x} u(x)\) and \(u_{xx}\) for \(\frac{\d^2}{\d x^2} u(x)\). For \(1 <p < \infty\) and \( m \in \N\), we define the Sobolev space
\begin{equation*}
  W^{m,p} (\R) : = \bigg \{ u \in L^p(\R) \; : \; \norm{u}_{W^{m,p}(\R)}:= \sum\limits_{ n \le m } \norm{\frac{\d^n}{\d x^{n}} u }_{L^p(\R)}  < \infty \bigg \},
\end{equation*}
where \(\frac{\d^n}{\d x^n} u\) are understood as weak derivatives, see e.g. \cite{AdamsRobertA2003Ss}. Moreover, we use the abbreviation \(H^m(\R):= W^{m,2}(\R)\), write \(H^{-1}(\R)\) for the dual space of \(H^1(\R)\) and denote the dual paring by \(\langle \cdot, \cdot\rangle_{H^{-1}(\R), H^1(\R)}\). Weak convergence is denoted by the symbol $\rightharpoonup$, where the involved function spaces are not further specified if they are clear from the context.

\subsection{Particle systems}\label{subsec: particle system}

In this subsection we introduce the probabilistic setting, in particular, the \(N\)-particle system and its regularized version. To that end, let \((\Omega, \mathcal{F}, ( \mathcal{F}_t)_{t \ge 0 } , \P)\) be a complete probability space with right-continuous filtration \((\mathcal{F}_t)_{t \ge 0 } \) and \((B_t^{i}, t\ge 0)\), \( i=1, \ldots, N\), be independent one-dimensional Brownian motions.

\medskip

Throughout the entire paper we make the following assumptions on the interaction force kernel~$k$ and the initial condition~\(\rho_0\) of the interacting particle system.

\begin{assumption}\label{ass: initial condition}
  The interaction force kernel \(k \colon \R \to \R\) satisfies
  \begin{equation*}
    k \in  L^\infty(\R)
  \end{equation*}
  and the initial condition \(\rho_0 \colon \R \to \R \) fulfills
  \begin{equation*}
    \rho_0 \in L^1(\R) \cap L^\infty(\R) \cap L^1(\R, |x| \Id x ) ,\quad
    \rho_0 \ge 0 ,\quad\text{and}\quad
    \int_{\R} \rho_0(x) \Id x = 1 .
  \end{equation*}
\end{assumption}

The \(N\)-particle system \(\mathbf{X}_t^N:= (X_t^{1}, \ldots,X_t^{N})\) is given by
\begin{equation}\label{eq: particle_system}
  \Id X_t^{i} = -\frac{1}{N} \sum\limits_{j=1}^N k(X_t^{i} -X_t^{j}) \Id t + \sigma \Id B_t^{i}, \quad i=1,\ldots,N , \; \; \mathbf{X}_0^N \sim \overset{N}{\underset{i=1}{\otimes}} \rho_0,
\end{equation}
where \(\sigma >0\) is the diffusion parameter and \( \mathbf{X}_0^N\) is independent of the Brownian motions \((B_t^{i}, t\ge 0)\), \( i=1, \ldots, N\). In the limiting case when \(N \to \infty\), the particle system~\eqref{eq: particle_system} induces the following i.i.d. sequence \(\mathbf{Y}_t^N := (Y_t^{1}, \ldots,Y_t^{N})\) of mean-field particles
\begin{equation}\label{eq: mean_field_trajectories}
  \Id Y_t^{i} = -  (k* \rho_t)(Y_t^{i}) \Id t + \sigma \Id B_t^{i}, \quad i=1,\ldots,N , \; \; \mathbf{Y}_0^N=\mathbf{X}_0^N,
\end{equation}
where \(\rho_t:= \rho(t,\cdot)\) denotes the probability density of any of the i.i.d. random variables \(Y_t^{i}\).

To introduce the regularized versions of \eqref{eq: particle_system} and \eqref{eq: mean_field_trajectories}, we take a smooth approximation \((k^\epsilon, \epsilon > 0)\) of \(k\). The regularized microscopic \(N\)-particle system \(\mathbf{X}_t^{N,\epsilon } := (X_t^{1,\epsilon}, \ldots,X_t^{N,\epsilon})\) is given by
\begin{equation}\label{eq: regularized_particle_system}
  \d X_t^{i,\epsilon} = -\frac{1}{N} \sum\limits_{j=1}^N k^\epsilon(X_t^{i,\epsilon} -X_t^{j,\epsilon}) \Id t + \sigma \Id B_t^{i}, \quad i=1,\ldots,N , \; \; \mathbf{X}_0^{N,\epsilon} \sim \overset{N}{\underset{i=1}{\otimes}} \rho_0,
\end{equation}
and the regularized mean-field trajectories \(\mathbf{Y}_t^{N,\epsilon} := (Y_t^{1,\epsilon}, \ldots,Y_t^{N,\epsilon})\) by
\begin{equation}\label{eq: regularized_mean_field_trajectories}
  \d Y_t^{i,\epsilon} = -  (k^\epsilon* \rho_t^\epsilon)(Y_t^{i,\epsilon}) \Id t + \sigma \Id B_t^{i}, \quad i=1,\ldots,N , \; \; \mathbf{Y}_0^{N,\epsilon}=\mathbf{X}_0^{N,\epsilon},
\end{equation}
where \(\rho_t^\epsilon:=\rho^\epsilon(t,\cdot)\) denotes the probability density of any of the i.i.d. random variables \(Y_t^{i,\epsilon}\).

Moreover, for \(i=1, \ldots,N\), it is convenient to denote the regularized interaction force \(K_i^\epsilon \colon \R^{N} \to \R\) as
\begin{equation}\label{eq: K_N}
  K_i^\epsilon(x_1,\ldots,x_N) := -\frac{1}{N} \sum\limits_{j=1}^N k^\epsilon(x_i-x_j), \quad (x_1,\ldots,x_N) \in \R^{N},
\end{equation}
and the mean-field interaction force \(\overline{K_{t,i}^\epsilon}: \R^{N} \to \R \) as
\begin{equation}\label{eq: averaged_K_N}
  \overline{K_{t,i}^\epsilon}(x_1,\ldots,x_N):=- (k^\epsilon*\rho_t^\epsilon)(x_i), \quad (x_1,\ldots,x_N) \in \R^{N},
\end{equation}
where \(\rho_t^\epsilon \) is the law of \(Y_t^{i,\epsilon}\).

\subsection{Diffusion--aggregation equations}

The associated probability densities of the particle systems, introduced in Subsection~\ref{subsec: particle system}, satisfy non-linear, non-local partial differential equations (PDEs). Indeed, the particle system~\eqref{eq: mean_field_trajectories} induces the non-linear diffusion-aggregation equation
\begin{align}\label{eq: aggregation_diffusion_pde}
  \begin{cases}
  \frac{\d}{\d t} \rho(t,x) = \frac{\sigma^2}{2} \rho_{xx} (t,x) +  ((k*\rho)(t,x) \rho(t,x))_x \quad & \forall (t,x) \in  [0,T)  \times \R \\
  \; \; \, \rho(x,0) = \rho_0 &\forall x \in \R
  \end{cases}
\end{align}
and the regularized particle system~~\eqref{eq: regularized_mean_field_trajectories} the diffusion-aggregation equation
\begin{align}\label{eq: regularized_aggregation_diffusion_pde}
  \begin{cases}
  \frac{\d}{\d t} \rho^\epsilon(t,x) = \frac{\sigma^2}{2} \rho^\epsilon_{xx} (t,x) +  ((k^\epsilon*\rho^\epsilon)(t,x) \rho^\epsilon(t,x))_x \quad & \forall (t,x) \in   [0,T) \times \R  \\
  \; \; \, \rho^\epsilon(x,0) = \rho_0 &\forall x \in \R
  \end{cases} . 
\end{align}

Note that we use \(\rho_t\) and \(\rho_t^\epsilon \) for the solutions of the PDEs~\eqref{eq: aggregation_diffusion_pde} and \eqref{eq: regularized_aggregation_diffusion_pde} as well as for the probability densities of the particle systems \eqref{eq: mean_field_trajectories} and \eqref{eq: regularized_mean_field_trajectories}, respectively, since these objects coincide by the superposition principle, see \cite{barburoeckner2020}, in combination with existence results of densities for the considered SDEs, see~\cite{romito2018}.

For the partial differential equations~\eqref{eq: aggregation_diffusion_pde} and \eqref{eq: regularized_aggregation_diffusion_pde} we rely on the concept of weak solutions, which we recall in the next definition.

\begin{definition}[Weak solutions]\label{def: weak_solution}
  Fix \(\epsilon > 0\) and \(T>0\). We say \(\rho^\epsilon \in L^2([0,T];H^1(\R)) \cap L^\infty([0,T];L^2(\R))\) with \( \frac{\d}{\d t} \rho^\epsilon \in L^2([0,T];H^{-1}(\R))\) is a weak solution of \eqref{eq: regularized_aggregation_diffusion_pde} if, for every \(\eta \in L^2([0,T];H^1(\R))\),
  \begin{equation}\label{eq: regularized_weak_solution}
    \int\limits_0^T \left\langle \frac{\d}{\d t} \rho^\epsilon_t, \eta \right\rangle_{H^{-1}(\R), H^1(\R)} \Id t = - \int\limits_0^T \int_\R \left ( \frac{\sigma^2}{2} \rho^\epsilon_{x} (t,x) + (k^\epsilon*\rho^\epsilon)(t,x) \rho^\epsilon(t,x) \right) \eta_x \Id x \Id t
  \end{equation}
  and \(\rho^\epsilon(0,\cdot) = \rho_0\). Note that \(\rho^\epsilon \in L^2([0,T];H^1(\R))\) with \( \frac{\d}{\d t} \rho^\epsilon \in L^2([0,T];H^{-1}(\R))\) implies \(\rho^\epsilon \in C([0,T];L^2(\R))\), see \cite[Chapter~5.9]{EvansLawrenceC2015Pde}. Similarly, we say that \(\rho \in L^2([0,T];H^1(\R)) \cap L^\infty([0,T];L^2(\R))\) with \( \frac{\d}{\d t} \rho \in L^2([0,T];H^{-1}(\R))\) is a weak solution of \eqref{eq: aggregation_diffusion_pde} if \eqref{eq: regularized_weak_solution} holds with the interaction force kernel \(k\) instead of its approximation \(k^\epsilon\).
\end{definition}

By the regularity of the solution in Definition~\ref{def: weak_solution} we can actually weaken the assumption on \(\eta\) in equation~\eqref{eq: regularized_weak_solution} to \(\eta \in C([0,T];C_c^\infty(\R))\).

\begin{remark}
  The divergence structure of the PDEs \eqref{eq: aggregation_diffusion_pde} and \eqref{eq: regularized_aggregation_diffusion_pde}, respectively, implies mass conservation/the normalisation condition
  \begin{equation*}
    1 = \int_\R \rho_t(x) \Id x = \int_\R \rho^\epsilon_t(x) \Id x, \quad  t\in [0,T],
  \end{equation*}
  under Assumption~\ref{ass: initial condition}. This is an immediate consequence by plugging in a cut-off sequence, see~\cite[Lemma~8.4]{BrezisHaim2011FaSs}, which converges to the constant function \(1\) as a test function in~\eqref{eq: regularized_weak_solution}.
\end{remark}

\subsection{Outline of the method}\label{subsec: sketch}

The method of the present paper originated from the approach of D. Lazarovici and P. Pickl, developed for the {V}lasov--{P}oisson system in~\cite{lazarovici2017mean}. It is based on the coupling method~\cite{snitzman_propagation_of_chaos} and a regularization of \(k\) to \(k^\epsilon \). A key insight of D. Lazarovici and P. Pickl is to prove the convergence in probability with an arbitrary large algebraic rate and algebraic cut-off parameter \(\epsilon \sim N^{-\beta}, \;  \beta >  0\), instead of comparing the trajectories \(\mathbf{X}^{N,\epsilon}\) and \(\mathbf{Y}^{N,\epsilon}\) in Wasserstein distance or in \(L^2\)-norm, as for instance done in \cite{snitzman_propagation_of_chaos,CarrilloChoi2014}. More precisely, for \(\alpha \in (0,1/2)\), \(\beta \le \alpha\) and arbitrary \(\gamma > 0\), we shall show that
\begin{equation*}
  \P\left(\sup\limits_{ t\in [0,T]} |\mathbf{X}_t^{N,\epsilon}-\mathbf{Y}_t^{N,\epsilon}|_\infty \ge N^{-\alpha}\right) \le C(\gamma) N^{-\gamma}, \quad \mathrm{for \; each}  \; N \ge N_0.
\end{equation*}
To implement this strategy and to achieve the aforementioned result, we proceed as follows:
\begin{enumerate}
  \item We start with a PDE analysis of the diffusion-aggregation equations~\eqref{eq: aggregation_diffusion_pde} and~\eqref{eq: regularized_aggregation_diffusion_pde}, that is, we prove the well-posedness of the non-local, non-linear PDEs~\eqref{eq: aggregation_diffusion_pde} and~\eqref{eq: regularized_aggregation_diffusion_pde}, together with an \(L^\infty([0,T];L^\infty(\R))\)-bound on the solution \(\rho^{\epsilon}\), which is uniform in \(\epsilon \). These results can be obtained via standard PDE techniques such as a compactness method, Aubin--Lions lemma, which provides strong convergence, and a Moser type iteration, see Section~\ref{sec: existence_of_pdes}. 
 The uniform bound allows us to have a trade-off between the irregularity of the interaction force kernel and the regularity of the solution \(\rho^\epsilon\).

  \item The main idea of D. Lazarovici and P. Pickl was to recognize that even though the interaction force kernel is not globally Lipschitz continuous, the approximation \(k^\epsilon\) satisfies a local Lipschitz bound of order \(\epsilon^{-1}\) (in dimension \(d\) of order \(\epsilon^{-d}\)) for \(|x-y| \le 2\epsilon\), i.e.
  \begin{equation}\label{eq: sketch_loc_lipschitz}
    |k^\epsilon(x)-k^\epsilon(y)|\le l^\epsilon(y) |x-y|.
  \end{equation}
  Let us emphasize that the bound depends only on the point~\(y\). Hence, the above inequality seems like a Taylor expansion around the point \(y\), where the second order term is missing. Consequently, the bound cannot be achieved by a simple application of the mean-value theorem.

  We will assume that the interaction force kernel~\(k\) satisfies~\eqref{eq: sketch_loc_lipschitz}, see Assumption~\ref{ass: loc_lip_bound} below, and present various examples of such kernels in Section~\ref{sec: local_lipschitz_bound}. We refer to~\cite{lazarovici2017mean,CarrilloChoiSalem2019,HuangHiuLius2019} for further models with interaction force kernels satisfying $\eqref{eq: sketch_loc_lipschitz}$. In general, whether \eqref{eq: sketch_loc_lipschitz} holds true entirely depends on the interaction force kernel of the considered model, in particular, on the order of discontinuity/singularity of the kernel. Hence, as rule of thumb, if the discontinuity/singularity is of order \(\epsilon^{-d+1}\) in a \(d\)-dimensional setting, then the local Lipschitz bound assumption can be satisfied.

  \item We need to derive a law of large numbers, see Section~\ref{sec: law_of_large_numbers}. This allows us to treat every involved object with regard to its expectation on a set with high probability, which enables us to take advantage of the obtained regularity of \(\rho^\epsilon\) in Step (1). Unsurprisingly, we need i.i.d. objects to apply the derived law of large numbers. In the present case these objects are going to be the processes \((Y_t^{i,\epsilon}, t \ge 0) \) for \(i \in \N\). Moreover, we would like to emphasize the importance of Step~(2) at this moment and the crucial fact that \(l^\epsilon(y)\) only depends on the point \(y\). Replacing in inequality~\eqref{eq: sketch_loc_lipschitz} the point \(y\) with the process \(Y_t^{i,\epsilon}\) and \(x \) with the process \(X_t^{i,\epsilon}\), we see that \(l^\epsilon\) on the right-hand side of~\eqref{eq: sketch_loc_lipschitz} is depending on the i.i.d. process \(Y_t^{i,\epsilon}\). Consequently, we can rely on the law of large numbers, Proposition~\ref{prop: law_of_large_numbers}.
 
  \item Finally, let us demonstrate how to apply the previous steps to derive propagation of chaos in probability but leaving out the technical difficulties. To that end, for some \(\alpha \in (0,1/2)\) and \(\delta >0\), we define an auxiliary process
  \begin{equation*}
    J_t^N := \min \bigg( 1, N^\alpha|\mathbf{X}^{N,\epsilon}-\mathbf{Y}^{N,\epsilon}|_\infty + N^\delta \bigg).
  \end{equation*}
  This process seems to control the difference \(|\mathbf{X}^{N,\epsilon}-\mathbf{Y}^{N,\epsilon}|_\infty\) in the limit \(N\to \infty\) with weight \(N^\alpha\). Furthermore, the minimum is no restriction, since we only want to show convergence to zero in probability, and we notice that, if \(N^\alpha|\mathbf{X}^{N,\epsilon}-\mathbf{Y}^{N,\epsilon}|_\infty\) is too big, the process stays constant one and the time derivative is zero. Therefore, we heuristically obtain
  \begin{align*}
    &\frac{\dd}{\dd t} (N^\alpha|\mathbf{X}^{N,\epsilon}-\mathbf{Y}^{N,\epsilon}|_\infty + N^{-\delta})\\
    &\quad\le N^\alpha \sup \limits_{i=1,\ldots,N} |K_i^{\epsilon}(\mathbf{X}_t) -\overline{K_{t,i}^{\epsilon}}(\mathbf{Y}_t) | \\
    &\quad\le N^\alpha \sup \limits_{i=1,\ldots,N} |K_i^{\epsilon}(\mathbf{X}_t) - K_i^{\epsilon}(\mathbf{Y}_t)) | + N^\alpha \sup \limits_{i=1,\ldots,N} |K_i^{\epsilon}(\mathbf{Y}_t) -\overline{K_{t,i}^{\epsilon}}(\mathbf{Y}_t) |.
  \end{align*}
  The last term depends on the i.i.d. particles \((Y_t^{i},i=1, \ldots, N)\) and can be estimated via the law of large numbers, Proposition~\ref{prop: law_of_large_numbers}, with a rate of \(N^{-\delta-\alpha}\). For the first term we can use the local Lipschitz bound (having in mind that the particles are close because of the minimum in the process) to complete a Gronwall argument. As mentioned before, the crucial point in this step is the fact that the local Lipschitz bound only depends on the i.i.d. particles \(\mathbf{Y}^{N,\epsilon}\) and not on the particles system \(\mathbf{X}^{N,\epsilon}\). This allows us to exchange the local Lipschitz bound \(\frac{1}{N} \sum\limits_{j=1}^N  l^\epsilon(Y_t^{i,\epsilon}-Y_t^{j,\epsilon})\) with its conditional expectation \(l^\epsilon*\rho_t^\epsilon(Y_t^{i,\epsilon})\). Using the regularity properties, obtained from the PDE analysis in Step~(1), we can bound \(\norm{l^\epsilon*\rho_t^\epsilon}_{L^\infty([0,T];L^\infty(\R))}\). Hence, we conclude that
  \begin{align*}
    \frac{\dd}{\dd t} (  N^\alpha|\mathbf{X}^{N,\epsilon}-\mathbf{Y}^{N,\epsilon}|_\infty+ N^{-\delta} )
    \le C (N^\alpha|\mathbf{X}^{N,\epsilon}-\mathbf{Y}^{N,\epsilon}|_\infty + N^{-\delta}).
  \end{align*}
  Applying Gronwall's lemma completes the proof. We remark that we implicitly used the fact that the law of large numbers holds for large \(N\in \N\) and, consequently, the above Gronwall inequality only holds in the limit \(N \to \infty\). In the actual proof we will use a version of the process \(J_t^N\) which is multiplied by an exponential, which just leads to a rewriting of the above Gronwall argument.
\end{enumerate}

The remaining of the present paper is devoted to establish Step (1)-(4) with all technical details for bounded interaction force kernels.

\section{Well-posedness and uniform bounds for the PDEs}\label{sec: existence_of_pdes}

In this section we prove well-posedness of the PDEs \eqref{eq: aggregation_diffusion_pde} and \eqref{eq: regularized_aggregation_diffusion_pde}, show the convergence of the solutions \((\rho^\epsilon, \epsilon > 0)\) to \(\rho\) in the weak topology, and provide regularity results as well as uniform bounds for \((\rho^\epsilon, \epsilon > 0)\) and \(\rho\), which are required for propagation of chaos result in probability established later in Section~\ref{sec: convergence_in_probability}. We start by introducing an assumption on the approximation sequence \((k^\epsilon, \epsilon> 0)\) of interaction force kernels.

\begin{assumption}\label{ass: kernel_smooth_convergence}
  Let \((k^\epsilon, \epsilon>0)\) be a sequence, which satisfies the following:
  \begin{enumerate}[label=(\roman*)]
    \item For each \(\epsilon>0\) the interaction force kernel \(k^\epsilon \in C^2(\R)\);
    \item For each \(\epsilon>0\) we have \(\norm{k^\epsilon}_{L^\infty(\R)} \le C \norm{k}_{L^\infty(\R)} < \infty\);
    \item We have \(\lim\limits_{\epsilon \to 0} k^\epsilon = k \; \, \text{a.e.}\)
  \end{enumerate}
\end{assumption}

For the non-linear, non-local PDE~\eqref{eq: regularized_aggregation_diffusion_pde} we notice that, by Young's inequality, we obtain the following \(L^\infty(\R)\)-bound
\begin{equation}\label{eq: l_infinity_bound_of_interaction_force}
  |(k^\epsilon*\rho^{\epsilon})(t,x)| \le \norm{k^\epsilon}_{L^\infty(\R)} \norm{\rho^\epsilon_t}_{L^1(\R)}
  \le C \norm{k}_{L^\infty(\R)} .
\end{equation}
Hence, \(k^\epsilon *\rho\) is uniformly bounded in \(\epsilon>0\) on \( [0,T] \times \R \). The same statement holds for \(k*\rho\). Consequently, the convolution term is bounded and we expect the existence of a weak solution to the PDEs~\eqref{eq: aggregation_diffusion_pde} and~\eqref{eq: regularized_aggregation_diffusion_pde}.
  
\begin{theorem}\label{theorem: existence_regularized_solutions}
  Suppose Assumption~\ref{ass: initial condition}. Then, for each \(T>0\) and \( \epsilon > 0\) there exists a unique non-negative weak solution \(\rho^\epsilon \in L^2([0,T];H^1(\R)) \cap L^\infty([0,T];L^2(\R))\) with \( \frac{\d}{\d t} \rho^\epsilon  \in L^2([0,T];H^{-1}(\R))\) to the regularized PDE~\eqref{eq: regularized_aggregation_diffusion_pde} in the sense of Definition~\ref{def: weak_solution}. Moreover, the estimate
  \begin{equation}\label{eq: regularized_solutions_estimate}
    \norm{\rho^\epsilon}_{L^\infty([0,T];L^2(\R))} + \norm{\rho^\epsilon}_{L^2([0,T];H^1(\R))}  + \norm{ \frac{\d}{\d t} \rho^\epsilon}_{L^2([0,T];H^{-1}(\R))}
    \le C(T) \norm{\rho_0}_{L^2(\R)}
  \end{equation}
  holds for all \(\epsilon >0\).
\end{theorem}

\begin{proof}
  Let us explain the main idea of the existence proof. We consider the associated McKean--Vlasov process
  \begin{align*}
    \begin{cases}
    \Id Y_t^{\epsilon} &= -  (k^\epsilon* \rho_t^\epsilon)(Y_t^{\epsilon}) \Id t + \sigma \Id B_t^{1},  \; \; Y_0 \sim \rho_0, \\
    \rho_t^\epsilon &= \mathrm{Law}(Y_t^{\epsilon})
    \end{cases}
  \end{align*}
  for the initial data \(\rho_0\). Then, by \cite[Proposition~2]{mishura2020}, the aforementioned SDE has a unique strong solution and, by \cite[Proposition~3.1]{romito2018}, it has a density \( (\rho_t^\epsilon,t\ge 0)\). Now, fix \(\rho_t^\epsilon\) and consider the solution \(\tilde{\rho}^\epsilon =( \tilde{\rho}^\epsilon_t,t\ge0)\) to the linearized parabolic PDE
  \begin{align*}
    \begin{cases}
     \frac{\d}{\d t} \tilde{\rho}^\epsilon(t,x) = \frac{\sigma^2}{2} \rho^\epsilon_{xx} (t,x) +  ((k^\epsilon*\rho^\epsilon)(t,x) \tilde{\rho}^\epsilon(t,x))_x \quad & \forall (t,x) \in [0,T) \times \R  \\
    \; \; \, \tilde{\rho}^\epsilon(x,0) = \rho_0 &\forall x \in \R
    \end{cases} .
  \end{align*}
  By standard second order parabolic PDE theory, we know that the aforementioned PDE is well-posed and
  \begin{equation*}
    \tilde{\rho}^\epsilon \in L^2([0,T];H^1(\R)) \cap L^\infty([0,T];L^2(\R)), \quad  \frac{\d}{\d t} \tilde{\rho}_t^\epsilon \in L^2([0,T];H^{-1}(\R)),
  \end{equation*}
  with the estimate~\eqref{eq: regularized_solutions_estimate}. Applying the superposition principle \cite[Theorem~4.1]{barburoeckner2020}, we find a weak solution to
  \begin{equation*}
    \Id \tilde{Y}_t^{\epsilon} = -  (k^\epsilon* \rho_t^\epsilon)(\tilde{Y}_t^{\epsilon}) \Id t + \sigma \Id B_t^{1}, \; \; \tilde{Y}_0^N \sim \rho_0,
    \quad t \in [0,T],
  \end{equation*}
  with \(\mathrm{Law}(\tilde{Y}_t^\epsilon)= \tilde{\rho}_t^\epsilon  \Id x\). Since strong uniqueness holds for the above SDE, we have \(\tilde{Y}^\epsilon  = Y^\epsilon \). By the Yamada--Watanabe theorem~\cite[Chapter~5, Proposition~3.20]{KaratzasIoannis2009Bmas} this implies uniqueness in law and therefore
  \begin{equation*}
    \tilde{\rho}_t^\epsilon  \Id x = \rho_t^\epsilon  \Id x, \quad t \in [0,T],
  \end{equation*}
  in the sense of measures. Hence, \(\tilde{\rho}_t^\epsilon =  \rho_t^\epsilon \; \P\text{-a.s.} \) for all \(t \in [0,T]\) and \(\rho^\epsilon\) has the desired regularity.
\end{proof}

\begin{lemma}\label{lemma: uniform_bound_x2}
  Fix \(\epsilon > 0\) and suppose Assumption~\ref{ass: initial condition}. Moreover, consider a solution \(\rho^\epsilon\) of the regularized diffusion-aggregation equation~\eqref{eq: regularized_aggregation_diffusion_pde} with initial data \(\rho_0  \), which by Theorem~\ref{theorem: existence_regularized_solutions} exists. Then, we have the following uniform bound
  \begin{equation*}
    \int_{\R} |x| \rho^\epsilon(t,x) \Id x \le \int_{\R}|x| \rho_0(x) \Id x
    + C(T) \norm{\rho_0}_{L^2(\R)} +CT
	+CT \norm{k}_{L^\infty(\R)} 
  \end{equation*}
  for all \(t \ge 0 \), which depend only upon \(\int_{\R} (1+|x|)\rho_0(x) \Id x \) and \(T\). Therefore, the function \(t \mapsto  \int_{\R} |x| \rho^\epsilon(t,x) \Id x \) is bounded in \( L^{\infty}([0,T];\R) \).
\end{lemma}

\begin{proof}
  The core idea is to use \(|x|\) as a test function. To that end, we take a sequence of radial antisymmetric functions \((g_n, n \in \N)\) with \(g_n \in C_c^2(\R)\) for all \( n \in \N \), such that \(g_n\) grows to \(|x|\) as \(n \to \infty\) and \(\frac{\d}{\d x} g_n\) is uniformly bounded in \(n \in \N\). More precisely, we choose
  \begin{align*}
    \chi_{n}(x) :=
    \begin{cases}
    |x| \quad   &|x| \ge \frac{1}{n}  \\
    -n^3\frac{x^4}{8}  + n \frac{3x^2}{4} + \frac{3}{8n} \quad & |x| \le \frac{1}{n}
    \end{cases}
  \end{align*}
  and let \((\zeta_n, n \in \N)\) be a sequence of compactly supported cut-off function converging to one with vanishing derivatives in the limit \(n \to \infty\), see \cite[Lemma~8.4]{BrezisHaim2011FaSs}. Define \(g_n := \chi_n \zeta_n\). We notice that \(\chi_n\) has bounded derivative with support in the unit interval and the derivative \( \frac{\d}{\d x} g_n\) has support in the annulus \([-2n,2n] \setminus (-n,n)\). Then, for \(\varphi \in \testfunctions{0,T}\) we obtain
  \begin{align*}
    &\int_{0}^T \int_{\R} g_n(x) \rho^\epsilon(t,x)  \frac{\d}{\d t} \varphi(t) \Id x \Id t \\
    &\quad= \int_{0}^T \int_{\R} \bigg( \frac{\sigma^2}{2} \frac{\d}{\d x} \rho^\epsilon(t,x)  \frac{\d}{\d x} g_n (x) + (k^\epsilon * \rho^\epsilon) \frac{\d}{\d x} g_n(x)  \rho^\epsilon(t,x) \bigg) \varphi(t) \Id x \Id t.
  \end{align*}
  Furthermore, for \(t_1, t_2 \in [0,T] \) we have
  \begin{equation*}
    \left |\int_{\R} g_n(x) \rho^\epsilon(t_1,x) \Id x -\int_{\R} g_n(x) \rho^\epsilon(t_2,x)  \Id x \right|
    \le \norm{g_n}_{L^2(\R)} \norm{ \rho^\epsilon(t_1,\cdot)-\rho^\epsilon(t_2,\cdot)}_{L^2(\R)} .
  \end{equation*}
  Therefore, \(\rho^\epsilon \in C([0,T];L^2(\R)) \) implies that \(t \mapsto \int_{\R} g_n(x) \rho^\epsilon(t,x) \Id x\) is continuous for each \(n \in \N\). Then, the fundamental lemma of calculus of variations, mass conservation and~\eqref{eq: regularized_solutions_estimate} imply
  \begin{align*}
    &\int_{\R} g_n(x) \rho^\epsilon(t,x) \Id x  \\
    &\quad= \int_{\R} g_n(x) \rho_0 (x) \Id x -
    \int_0^t\int_{\R}  \frac{\sigma^2}{2} \frac{\d}{\d x}\rho^\epsilon(s,x)  \frac{\d}{\d x} g_n (x) + \chi (k^\epsilon * \rho^\epsilon)(s,x) \frac{\d}{\d x} g_n (x) \rho^\epsilon(s,x) \Id x \Id s\\
    &\quad\le \,  \int_{\R}|x| \rho_0(x)\Id x
    +  \frac{C \sigma^2}{2}  \int_0^T  \int\limits_{-1}^{1}\left |\frac{\d}{\d x} \rho^\epsilon(s,x)\right | \Id x  \Id s \\
    &\quad \quad- \frac{\sigma^2}{2}  \int_0^t \int_{\R} \chi_n(x)  \frac{\d}{\d x} \zeta_n(x)  \frac{\d}{\d x}\rho^\epsilon(s,x)  \Id x \Id s \\
    &\quad\quad  + \bigg|  \int_0^t \int_{\R^2} k^\epsilon(x-y)  \frac{\d}{\d x} g_n (x) \rho^\epsilon(s,y) \Id y \rho^\epsilon(s,x)  \Id x \Id s  \bigg| \\
    &\quad\le  \,  \int_{\R}|x| \rho_0(x) \Id x
    + C  T^{\frac{1}{2}} \bigg(\int_0^T  \int_{\R} \left| \frac{\d}{\d x} \rho^\epsilon(s,x) \right|^2 \Id x  \Id s \bigg)^{\frac{1}{2}} + C \int_0^T  \int_{\R} \rho^\epsilon(s,x)  \Id x \Id s\\
    &\quad\quad+  \frac{\sigma^2}{2}  \int_0^t \int_{\R} \chi_n(x) \rho^\epsilon(s,x) \frac{\d^2}{\d x^2} \zeta_n(x)  \Id x \Id s \\
    &\quad\quad+ \bigg|  \int_0^t \int_{\R^2} k^\epsilon(x-y)  \frac{\d}{\d x} g_n (x) \rho^\epsilon(s,y) \Id y \rho^\epsilon(s,x)  \Id x \Id s  \bigg| \\
	&\quad\le  \,  \int_{\R}|x| \rho_0(x) \Id x
    + C(T) \norm{\rho_0}_{L^2(\R)} +CT + \frac{C}{n^2} \int_0^T \int\limits_{-2n}^{2n} \chi_n(x) \rho^\epsilon(s,x)\Id x \Id s  \\
    &\quad\quad +  C  \norm{k}_{L^\infty(\R)}\int\limits_0^T  \int_{\R^2}  \rho^\epsilon(s,y) \Id y \rho^\epsilon(s,x)  \Id x \Id s  \\
	&\quad\le \, \int_{\R}|x| \rho_0(x) \Id x
    + C(T) \norm{\rho_0}_{L^2(\R)} +CT + \frac{CT}{n}
	+C T  \norm{k}_{L^\infty(\R)}  . 
  \end{align*}
  Applying Fatou's lemma proves the lemma. 
\end{proof}

\begin{lemma}\label{lemma: aubin_lion_space}
  Let \((f_n,n \in \N)\) be a sequence in \(L^2([0,T];H^1(\R))\). If the sequence satisfies
  \begin{enumerate}[label=(\roman*)]
    \item \(\norm{f_n}_{L^2([0,T];H^{1}(\R))} \le C \),
    \item \(\norm{ \frac{\d}{\d t} f_n}_{L^2([0,T];H^{-1}(\R))}\le C \),
    \item \(\sup\limits_{t \in [0,T]} \int_{\R} |x| |f_n(t,x)| \Id x \le C \),
  \end{enumerate}
  for some constant \(C>0\), then \((f_n,n\in \N)\) is relative compact in \(L^p([0,T];L^p(\R))\) for all \(p \in [1,2]\).
\end{lemma}

\begin{proof}
  Let us denote by \(B_{R}\) the ball with radius \(R\) and center \(0\). Then, by the Rellich-Kondrachov theorem we have that the embedding \(H^1(B_R) \hookrightarrow L^2(B_R)\) is compact. Hence, by Aubin--Lions lemma~\cite[Chapter~3, Proposition~1.3]{showalter1997}, \((f_n, n \in \N)\) is relative compact in \(L^2( [0,T];L^2(B_R))\). Since the above spaces is of finite measure, we obtain the relative compactness of \((f_n, n \in \N)\) in \(L^p([0,T];L^p(B_R))\) for all \(p \in [1,2]\). Note that we can extract only one subsequence for all above spaces, which will depend on the radius \(R\).

  In order to get rid of the dependency on \(R\), we perform a Cantor diagonal argument to extract a subsequence \((f_{n_k}, k \in \N)\) such that
  \begin{equation*}
    \lim\limits_{k \to \infty} \norm{f_{n_k}-f}_{L^p([0,T];L^p(B_R))} = 0
  \end{equation*}
  for some limit point \(f \in L^p([0,T];L^p(B_R))\) and all \(p \in [1,2]\), \(R \in \N\). Furthermore, using again a Cantor's diagonal argument, we can assume that \((f_{n_k}, k \in \N)\) converges almost everywhere to \(f\). It remains to prove that \(f \in L^p([0,T];L^p(\R))\) and \(f_{n_k} \to f \in L^p([0,T];L^p(\R))\) as $k\to \infty$. First, by Fatou's lemma we obtain \(f \in L^p([0,T];L^p(\R))\) and
  \begin{equation*}
    \sup\limits_{t \in [0,T]} \int_{\R} |x| f(t,x) \Id  x
    \le \sup\limits_{n \in \N} \sup\limits_{t \in [0,T]}  \int_{\R} |x| f_{n}(t,x) \Id  x  \le C .
  \end{equation*}
  Second, we find
  \begin{align*}
    &\norm{f_{n_k} -f}_{L^p([0,T];L^p(\R))}^p\\
    &\quad= \int\limits_0^T \int_{\R} |f_{n_k}(t,x)-f(t,x)|^p \Id x \Id t  \\
    &\quad= \int\limits_0^T \int_{B_R} |f_{n_k}(t,x)-f(t,x)|^p \Id x \Id t + \int\limits_0^T \int_{B_R^{\mathrm{c}}} |f_{n_k}(t,x)-f(t,x)|^p \Id x \Id t  \\
    &\quad\le \int\limits_0^T \int_{B_R} |f_{n_k}(t,x)-f(t,x)|^p \Id x \Id t + \frac{1}{R} \sup\limits_{k \in \N} \int\limits_0^T \int_{B_R^{\mathrm{c}}} |x| |f_{n_k}(t,x)-f(t,x)|^p \Id x \Id t .
  \end{align*}
  Taking \(k \to \infty\) and then \(R \to \infty\), we find a subsequence \((f_{n_k}, k \in \N)\) which converges in \(L^p([0,T];L^p(\R))\) and, thus, the sequence \((f_n, n \in \N)\) is relative compact in \(L^p([0,T];L^p(\R))\).
\end{proof}

In the next theorem, we show that the approximation sequence \((\rho^\epsilon, \epsilon > 0)\) converges in the weak sense to a weak solution \(\rho\) of equation~\eqref{eq: aggregation_diffusion_pde}.

\begin{theorem}\label{theorem: existence_solution}
  Suppose Assumption~\ref{ass: initial condition}. Then, for each \(T>0\) there exists a subsequence \((\rho^{\epsilon_m}, m \in \N)\) such that \(\rho^{\epsilon_m} \rightharpoonup \rho \) as \(m \to \infty\) in \(L^2([0,T];H^1(\R))\). Furthermore, \(\rho \in L^2([0,T];H^1(\R)) \cap L^\infty([0,T];L^2(\R))\) with \( \frac{\d}{\d t} \rho \in L^2([0,T];H^{-1}(\R))\) is the unique non-negative weak solution of equation \eqref{eq: aggregation_diffusion_pde}, which satisfies
  \begin{equation}\label{eq: solutions_estimate}
    \norm{\rho}_{L^\infty([0,T];L^2(\R))} + \norm{\rho}_{L^2([0,T];H^1(\R))}  + \norm{ \frac{\d}{\d t} \rho}_{L^2([0,T];H^{-1}(\R))}
    \le C(T) \norm{\rho_0}_{L^2(\R)}.
  \end{equation}
  In addition, there exists a subsequence \((\rho^{\epsilon_m}, m \in \N)\) such that \(\rho^{\epsilon_m} \to \rho \) converges weakly as \(m \to \infty\) in \(L^1([0,T];L^1(\R))\).
\end{theorem}

\begin{proof}
  From \eqref{eq: regularized_solutions_estimate}, the Banach--Alaoglu theorem and the lower semi-continuity we obtain \eqref{eq: solutions_estimate} and a subsequence \((\rho^{\epsilon_m}, m \in \N)\) such that
  \begin{align*}
    \rho^{\epsilon_m} \rightharpoonup \rho\quad  & \mathrm{in}  \;  L^2([0,T];H^1(\R)), \\
     \frac{\d}{\d t} \rho^{\epsilon_m} \rightharpoonup  \frac{\d}{\d t} \rho   \quad  & \mathrm{in} \;  L^2([0,T];H^{-1}(\R)).
  \end{align*}
  Moreover, we have \(\rho \ge 0\) a.e. by Mazur's lemma \cite[Corollary~3.8]{BrezisHaim2011FaSs}. Next, we notice that the subsequence \((\rho^{\epsilon_m}, m \in \N)\) fulfills Lemma~\ref{lemma: aubin_lion_space}. Consequently, without renaming the subsequence we conclude
  \begin{equation}\label{eq: L1_conv}
    \lim\limits_{m \to \infty} \norm{\rho^{\epsilon_m} - \rho}_{L^p([0,T];L^p(\R))} = 0
  \end{equation}
  for all \(p \in [1,2]\).
  Hence, it remains to show that we can take the limit in \eqref{eq: regularized_weak_solution}. From the above weak convergence it immediately follows
  \begin{align*}
    &\int\limits_0^T \left \langle \frac{\d}{\d t} \rho^{\epsilon_m}_t, \eta \right \rangle_{H^{-1}(\R), H^1(\R)} \Id t  \to \int\limits_0^T  \left \langle \frac{\d}{\d t} \rho_t, \eta\right \rangle_{H^{-1}(\R), H^1(\R)} \Id t,  \\
    &  \int\limits_0^T \int_\R  \rho^{\epsilon_m}_{x} (t,x) \eta_x(t,x) \Id x \Id t
    \to  \int\limits_0^T \int_\R \rho_{x}(t,x) \eta_x (t,x) \Id x \Id t
  \end{align*}
  for \(\eta \in L^2([0,T];H^1(\R))\) as \(m \to \infty\). We write the non-linear term as
  \begin{align}\label{eq: splitting_non_linear_term}
    &\int\limits_0^T \int_{\R} \rho^{\epsilon_m} (t,x) (k^{\epsilon_m}*\rho^{\epsilon_m})(t,x) \eta_x(t,x)  \Id x \Id t \nonumber \\
    &\quad=  \; \int\limits_0^T \int_{\R} (\rho^{\epsilon_m}-\rho) (k^{\epsilon_m}*\rho^{\epsilon_m})(t,x) \eta_x(t,x) \Id x \Id t\\
    &\quad\quad+\int\limits_0^T \int_{\R} \rho  ((k^{\epsilon_m}-k)*\rho)(t,x) \eta_x (t,x)\Id x \Id t\nonumber \\
    &\quad\quad+\int\limits_0^T \int_{\R} \rho (k^{\epsilon_m}*(\rho^{\epsilon_m}-\rho) ) (t,x) \eta_x (t,x)  \Id x \Id t
    +\int\limits_0^T \int_{\R} \rho (k*\rho)(t,x) \eta_x  (t,x) \Id x \Id t. \nonumber
  \end{align}
  For the first term we notice that it vanishes as \(m \to \infty\). Indeed, since \(|k^{\epsilon_m}*\rho^{\epsilon_m}| \le C \norm{k}_{L^\infty(\R)} \) and \(n_x \in L^2([0,T];L^2(\R))\), we have \((k^{\epsilon_m}*\rho^{\epsilon_m}) n_x \in L^2([0,T];L^2(\R)) \) uniform in \(\epsilon_m\) and, thus, \(\rho^{\epsilon_m} \rightarrow \rho \) in \(L^2([0,T];L^2(\R))\) implies
  \begin{equation*}
    \int\limits_0^T \int_{\R} (\rho^{\epsilon_m}-\rho)  (k^{\epsilon_m}*\rho^{\epsilon_m})(t,x) \eta_x(t,x) \Id x \Id t \to 0, \quad \text{as} \quad m \to \infty
  \end{equation*}
  by Hölder's inequality. 
  For the second term we use Assumptions~\ref{ass: kernel_smooth_convergence} to find
  \begin{align*}
    &\int\limits_0^T \int_{\R} \rho (t,x) ((k^{\epsilon_m}-k)*\rho)(t,x) \eta_x(t,x) \Id x \Id t \\
    &\quad=  \;  \int\limits_0^T \int_{\R}  \int_{\R}  \rho(t,x) (k^{\epsilon_m}-k)(x-y) \rho(t,y) \eta_x(t,x) \Id y \Id x \Id t \\
    &\quad\le  \;   \int\limits_0^T \int_{\R}  \int_{\R}  (\rho  \eta_x) (t,x)  \norm{k^{\epsilon_m}+k }_{L^\infty(\R)} \rho (t,y) \Id y  \Id x  \Id t  \\
    &\quad\le  \;   (C+1) \int\limits_0^T \int_{\R}  \int_{\R}  (\rho  \eta_x) (t,x) \norm{k}_{L^\infty(\R)} \rho (t,y) \Id y  \Id x  \Id t.
  \end{align*}
  The right-hand side is finite and, therefore, by the dominated convergence theorem and the almost everywhere convergence of \(k^\epsilon \to k\), the second term vanishes. For the third term we apply Young's inequality \cite[Theorem~4.2]{LiebElliottH2010A} and~\eqref{eq: L1_conv} to obtain
  \begin{align*}
    & \int\limits_0^T \int_{\R} \rho(t,x)  (k^{\epsilon_m}*(\rho^{\epsilon_m}-\rho)(t,x) )  \eta_x(t,x)   \Id x \Id t \\
    &\quad\le \;   \int\limits_0^T \norm{\rho^{\epsilon_m}        -\rho(t,\cdot)}_{L^1(\R)} \norm{k}_{L^\infty(\R)} \norm{\rho \eta_x(t,\cdot)}_{L^1(\R)} \Id t \\
    &\quad\le  \;  \norm{\eta}_{L^2([0,T];H^1(\R))} \norm{\rho}_{L^2([0,T];L^2(\R))} \norm{k}_{L^\infty(\R)}  \int\limits_0^T \norm{\rho^{\epsilon_m}-\rho(t,\cdot)}_{L^1(\R)} \Id t \\
    &\quad\rightarrow  \;  0,  \quad \mathrm{as} \;  m \to \infty .
  \end{align*}
  Consequently, taking the limit \(m \to \infty\) in \eqref{eq: splitting_non_linear_term}, we discover
  \begin{equation*}
    \lim\limits_{m \to \infty} \int\limits_0^T \int_{\R} \rho^{\epsilon_m} (t,x) (k^{\epsilon_m}*\rho^{\epsilon_m})(t,x) \eta_x(t,x) \Id x \Id t
    = \int\limits_0^T \int_{\R} \rho (k*\rho)(t,x) \eta_x(t,x)   \Id x \Id t
  \end{equation*}
  and therefore \(\rho \) is a weak solution. The uniqueness follows by simple \(L^2\)-estimates; see for instance \cite[Theorem~3.10]{CHAZELLE2017365} in the case of the Hegselmann--Krause model (notice that the proof of the uniqueness also works for \(\R\) and \(k\in L^\infty(\R) \)).
\end{proof}

\begin{remark}
  The uniqueness of the solution \(\rho\) actually implies that any subsequence convergences to the solution \(\rho\).
\end{remark}

\begin{lemma}\label{lemma: lp_bound_hk}
  Suppose Assumption~\ref{ass: initial condition}. Then, for any \(T>0\) the weak solutions \((\rho^\epsilon, \epsilon > 0)\) of \eqref{eq: regularized_aggregation_diffusion_pde} as well as the weak solutin \(\rho\) of \eqref{eq: aggregation_diffusion_pde} with initial condition \(\rho_0\) are bounded in \(L^\infty([0,T];L^p(\R))\). More precisely, we have, for all \(\epsilon > 0\),
  \begin{equation*}
    \norm{\rho^\epsilon}_{L^\infty([0,T];L^p(\R))} ,\norm{\rho}_{L^\infty([0,T];L^p(\R))} \le C(p,\sigma,T,\norm{k}_{L^\infty(\R)}) \norm{\rho_0}_{L^p(\R)}.
  \end{equation*}
\end{lemma}

\begin{proof}
  Without loss of generality we show the claim only for \(\rho\) and we also may assume that \(\rho\) is a smooth solution. Otherwise we mollify the initial data such that there exists a sequence of smooth solutions, which converge weakly in \(L^2([0,T];L^2(H^1(\R))) \) to \(\rho^\epsilon\) for each fix \(\epsilon > 0\). Applying the lower semi-continuity for each \(\epsilon > 0\) first and then the convergence result in Theorem~\ref{theorem: existence_solution} will prove the lemma.

  Multiplying \eqref{eq: aggregation_diffusion_pde} with \(\frac{p}{2(p-1)} \rho^{p-1}\), integrating by parts over \(\R\) and using inequality~\eqref{eq: l_infinity_bound_of_interaction_force}, we obtain
  \begin{align*}
    &\frac{1}{2(p-1)} \frac{\d}{\d t} \int_{\R} \rho^p(t,x) \Id x\\
    &\quad= \frac{p}{2(p-1)} \int_{\R}\frac{\d}{\d t} \rho(t,x) \rho^{p-1}(t,x) \Id x  \\
    &\quad=  \frac{p}{2(p-1)} \int_{\R} \bigg( \frac{\sigma}{2} \rho_{xx} (t,x) +  ((k*\rho)(t,x) \rho(t,x))_x  \bigg) \rho^{p-1}(t,x) \Id x  \\
    &\quad= \frac{p}{2} \int_{\R} - \frac{\sigma}{2} |\rho_x(t,x)|^2 \rho^{p-2}(t,x)  -  (k*\rho)(t,x) \rho_x \rho^{p-1}(t,x)  \Id x  \\
    &\quad= -\frac{\sigma^2}{p}\int_{\R} |(\rho^{p/2})_x(t,x)|^2 \Id x  - \int_{\R} (\rho^{p/2})_x(t,x) \rho^{p/2}(t,x) (k*\rho)(t,x)  \Id x  \\
    &\quad\le -\frac{\sigma^2}{p}\int_{\R} |(\rho^{p/2})_x(t,x)|^2 \Id x
    + \norm{k}_{L^\infty(\R)} \int_{\R} |(\rho^{p/2})_x(t,x)\rho^{p/2}(t,x)| \Id x \\ \
    &\quad\le -\frac{\sigma^2}{p}\int_{\R} |(\rho^{p/2})_x(t,x)|^2 \Id x
    + \norm{k}_{L^\infty(\R)} \int_{\R} \frac{\sigma^2}{2 p \norm{k}_{L^\infty(\R)}}  |(\rho^{p/2})_x(t,x)|^2 \\
    &\quad\quad + \frac{p\norm{k}_{L^\infty(\R)}}{2 \sigma^2} |\rho^p(t,x)| \Id x  \\
    &\quad\le -\frac{\sigma^2}{2p}\int_{\R} |(\rho^{p/2})_x(t,x)|^2 \Id x
    +\frac{p\norm{k}_{L^\infty(\R)}^2}{2 \sigma^2}  \int_{\R} |\rho^p(t,x)| \Id x  \\
    &\quad\le \frac{p\norm{k}_{L^\infty(\R)}^2}{2 \sigma^2}  \int_{\R} |\rho^p(t,x)| \Id x,
  \end{align*}
  where we used Young's inequality with \(\epsilon=\frac{\sigma^2}{\norm{k}_{L^\infty(\R)}p}\) in the sixth step. An application of Gronwall's inequality leads to
  \begin{equation*}
    \int_\R \rho^p(t,x) \Id x  \le C(p,\sigma,T,\norm{k}_{L^\infty(\R)}) \norm{\rho_0}_{L^p(\R1)} \quad \text{for all } t\in [0,T].
  \end{equation*}
\end{proof}

\begin{lemma}\label{lemma: infinity_bound_solution}
  Suppose Assumption~\ref{ass: initial condition}. Then, for each \(T>0\) there exists a constant \(C(\rho_0)\) such that, for all \(\epsilon>0\),
  \begin{equation*}
    \norm{\rho^\epsilon}_{L^\infty([0,T];L^\infty(\R))}, \norm{\rho}_{L^\infty([0,T];L^\infty(\R))} \le C(\rho_0)
  \end{equation*}
  holds for the weak solutions \((\rho^\epsilon,  \epsilon> 0)\) of \eqref{eq: regularized_aggregation_diffusion_pde} and for the weak solution \(\rho\) of \eqref{eq: aggregation_diffusion_pde}.
\end{lemma}

\begin{proof}
  As previously, we will only show the claim for \(\rho\) and we can assume that \(\rho\) is smooth.

  Set \(\rho_m:= \max(\rho-m,0)\) for some fix strictly positive \(m \in \R\) and let \(p > 2 \). For the sake of notational brevity we drop the depend of the involved on $(t,x)$. Multiplying \eqref{eq: aggregation_diffusion_pde} by \(\rho_m^{p-1}\) and integrating by parts, we obtain
  \begin{align*}
    \frac{1}{p} \frac{\d}{\d t} \int_{\R} \rho_m^p \Id x
    =& \int_{\R}  \bigg(\frac{\sigma^2}{2} \rho_{xx} + ((k*\rho) \rho)_x \bigg) \rho_m^{p-1} \Id x  \\
    =& - \int_{\R}   \frac{\sigma^2 (p-1)}{2} \rho_x (\rho_m)_x \rho_m^{p-2} -(p-1) (\rho_m)_x  \rho_m^{p-2} (k*\rho)  \rho  \Id x \\
    =& - \int_{\R}   \frac{\sigma^2 (p-1)}{2}  (\rho_m)_x^2 \rho_m^{p-2} -(p-1) \rho_m^{p-1} (\rho_m)_x (k*\rho)  \\
    &+m(p-1) (\rho_m)_x \rho_m^{p-2} (k*\rho) \Id x \\
    =&  - \frac{2 \sigma^2 (p-1)}{p^2} \int_{\R}  (\rho_m^{p/2})_x^2  \Id x
    - \frac{2(p-1)}{p}\int_{\R}  (\rho_m^{p/2})_x \rho_m^{p/2} (k*\rho) \Id x \\
    &+ \frac{2m(p-1)}{p} \int_{\R}  (\rho_m^{p/2})_x \rho_m^{p/2-1} (k*\rho) \Id x .
  \end{align*}
  In the next step we estimate the last two terms with Young's inequality. More precisely, we get
  \begin{align*}
    &2(p-1) \int_{\R}  \frac{1}{p}(\rho_m^{p/2})_x \rho_m^{p/2} (k*\rho) \Id x\\
    &\quad\le 2(p-1) \norm{k}_{L^\infty(\R)} \int_{\R}  \frac{1}{p}|(\rho_m^{p/2})_x| \,  |\rho_m^{p/2}| \Id x  \\
    &\quad\le 2(p-1) \int_{\R}  \frac{\sigma^2}{4p^2}|(\rho_m^{p/2})_x|^2 + \frac{ \norm{k}_{L^\infty(\R)}^2}{\sigma^2} |\rho_m^{p/2}|^2 \Id x  \\
    &\quad= \frac{(p-1)\sigma^2}{2p^2} \int_{\R}  |(\rho_m^{p/2})_x|^2 \Id x+ \frac{2(p-1) \norm{k}_{L^\infty(\R)}^2}{\sigma^2} \int_{\R}  |\rho_m^{p}| \Id x
  \end{align*}
  and
  \begin{align*}
    &2(p-1) \int_{\R}  \frac{1}{p}(\rho_m^{p/2})_x m \rho_m^{p/2-1} (k*\rho) \Id x \\
    &\quad\le 2(p-1) \int_{\R}  \frac{1}{p}|(\rho_m^{p/2})_x|\,  m  \norm{k}_{L^\infty(\R)} |\rho_m^{p/2-1}| \Id x \\
    &\quad\le 2 (p-1) \int_{\R}  \frac{\sigma^2}{4 p^2}|(\rho_m^{p/2})_x|^2 +   \frac{m^2 \norm{k}_{L^\infty(\R)}^2}{\sigma^2} |\rho_m^{p-2}| \Id x \\
    &\quad\le \frac{(p-1)\sigma^2}{2p^2} \int_{\R}  |(\rho_m^{p/2})_x|^2 \Id x
    + \frac{2(p-1) \norm{k}_{L^\infty(\R)}^2 m^2}{\sigma^2} \int_{\R}  |\rho_m^{p-2}| \Id x.
  \end{align*}
  Furthermore, we can estimate
  \begin{align*}
    \int_{\R}  |\rho_m^{p-2}| \Id x
    &= \int_{\R}  \indicator{\{m \le \rho \le m+1 \}} |\rho_m^{p-2}| +\indicator{\{\rho \ge m+1 \}}  |\rho_m^{p-2}|  \Id x \\
    &\le \int_{\R}  \indicator{\{m \le \rho \le m+1 \}} + |\rho_m^{p}| \Id x \\
    &\le  \frac{1}{m} \int_{\R}  \rho \Id x +  \int_{\R}  |\rho_m^{p}| \Id x \\
    &\le \frac{1}{m} +  \int_{\R}  |\rho_m^{p}| \Id x.
  \end{align*}
  Hence, we derive for the last term the following inequality
  \begin{align*}
    &2(p-1) \int_{\R}  \frac{1}{p}(\rho_m^{p/2})_x m \rho_m^{p/2-1} (k*\rho) \Id x \\
    &\quad\le \; \frac{(p-1)\sigma^2}{2p^2} \int_{\R}  |(\rho_m^{p/2})_x|^2 \Id x  +
    \frac{2(p-1) \norm{k}_{L^\infty(\R)}^2m}{\sigma^2} + \frac{2(p-1) \norm{k}_{L^\infty(\R)}^2m^2}{\sigma^2} \int_{\R}  |\rho_m^{p}| \Id x.
  \end{align*}
  Putting everything together we find
  \begin{align*}
    \frac{1}{p} \frac{\d}{\d t} \int_{\R}  \rho_m^p \Id x
    \le  & - \frac{\sigma^2 (p-1)}{p^2} \int_{\R}  (\rho_m^{p/2})_x^2  \Id x
    + \frac{2(p-1) \norm{k}_{L^\infty(\R)}^2m}{\sigma^2}\\
    &+ \left(  \frac{2(p-1) \norm{k}_{L^\infty(\R)}^2}{\sigma^2} +\frac{2(p-1) \norm{k}_{L^\infty(\R)}^2m^2}{\sigma^2}  \right)  \int_{\R}  |\rho_m^{p}| \Id x,
  \end{align*}
  from which we can conclude that
  \begin{equation*}
    \frac{\d}{\d t} \int_{\R}  \rho_m^p \Id x
    \le -\frac{\sigma^2}{2} \int_{\R}  (\rho_m^{p/2})_x^2  \Id x
    + \frac{2 \norm{k}_{L^\infty(\R)}^2p^2(m^2+1)}{\sigma^2}  \int_{\R}  |\rho_m^{p}| \Id x
    + \frac{2p^2 \norm{k}_{L^\infty(\R)}^2m}{\sigma^2}.
  \end{equation*}
  By the Gagliardo--Nirenberg--Sobolev interpolation inequality \cite[Theorem~12.87]{LeoniGiovanni2017Afci} and \cite{ASNSP_1959_3_13_2_115_0} on the whole space as well as Young's inequality with \(\tau = \frac{3 \sigma^2}{4} \) we have
  \begin{align}\label{eq: gns_l_infity_bound}
    \begin{split}
    \lambda^2 \norm{u}_{L^2(\R)}^2
    &\le C_{\mathrm{GNS}} \lambda^2 \norm{u}_{L^1(\R)}^{4/3} \norm{\nabla u}_{L^2(\R)}^{2/3}   \\
    &\le \frac{4}{\sigma^2 \sqrt{27}} C_{\mathrm{GNS}}^{3/2} \lambda^{3} \norm{u}_{L^1(\R)}^{2}  + \frac{\sigma^2}{4} \norm{\nabla u}_{L^2(\R)}^{2} ,
    \end{split}
  \end{align}
  where \(C_{\mathrm{GNS}}\) is the Gagliardo--Nirenberg--Sobolev constant in one dimension. For \(u=\rho_m^{p/2}\), \(C_1:= \frac{2\norm{k}_{L^\infty(\R)}^2(m^2+1)}{\sigma^2}\) and \(\lambda = \sqrt{C_1} p \) we obtain
  \begin{equation*}
    C_1 p^2  \int_{\R} |\rho_m^{p}| \Id x
    \le  \frac{4}{3 \sigma} C^{3/2} C_1^{3/2} p^3  \norm{\rho_m^{p/2} }_{L^1(\R)}^{2}  +  \frac{\sigma^2}{4} \norm{ (\rho_m^{p/2})_x }_{L^2(\R)}^{2} . 
  \end{equation*}
  Consequently, we have
  \begin{align*}
    \frac{\d}{\d t} \int_{\R} \rho_m^p \Id x
    \le & -\frac{1}{4} \sigma^2\int_{\R} (\rho_m^{p/2})_x^2  \Id x
    + p^3 \frac{4 C^{3/2}(2 \norm{k}_{L^\infty(\R)}^2( m^2+1))^{3/2} }{3\sigma^4}  \left(\int_{\R} |\rho_m^{p/2}| \Id x \right)^2\\
    &+ \frac{2p^2 \norm{k}_{L^\infty(\R)}^2m}{\sigma^2}  .
  \end{align*}
  Applying the above inequality~\eqref{eq: gns_l_infity_bound} with \(u=\rho_m^{p/2}\), \(\lambda = p\) and rearranging the terms we discover
  \begin{equation*}
    -\frac{\sigma^2}{4} \norm{ (\rho_m^{p/2})_x }_{L^2(\R)}^{2}
    \le - p^2  \int_{\R} |\rho_m^{p}| \Id x +  \frac{4}{3 \sigma} C^{3/2} p^3  \norm{\rho_m^{p/2} }_{L^1(\R)}^{2} ,
  \end{equation*}
  which then implies
  \begin{align*}
    &\frac{\d}{\d t} \int_{\R} \rho_m^p \Id x\\
    &\quad\le - p^2  \int_{\R} |\rho_m^{p}| \Id x + C(\sigma,m, \norm{k}_{L^\infty(\R)}) p^3   \left(\int_{\R} \rho_m^{p/2} \Id x \right)^2 + \frac{2p^2 \norm{k}_{L^\infty(\R)}^2m}{\sigma^2}  \\
    &\quad\le - p^2  \int_{\R} |\rho_m^{p}| \Id x + C(\sigma,m, \norm{k}_{L^\infty(\R)}) p^4 \left(\int_U \rho_m^{p/2} \Id x \right)^2 + C(\sigma,m, \norm{k}_{L^\infty(\R)}) p^2
  \end{align*}
  for some non-negative constant \(C(\sigma,m, \norm{k}_{L^\infty(\R)})\).  

  Let us define
  \begin{equation*}
    w_j(t) := \int_{\R} \rho_m^{2^j}(t,x) \Id x
    \quad \text{and}\quad
    S_j:= \sup\limits_{t \in [0,T)} w_j(t),
  \end{equation*}
  for \(j \in \N\). Then, for \(p = 2^j\) the above inequality can be written as
  \begin{align*}
    \frac{\d}{\d t} w_j(t)
    &\le -2^{2j}  w_j(t)  + 2^{2j} ( C(\sigma,m,\norm{k}_{L^\infty(\R)})2^{2j}  w_{j-1}^2(t)  + C(\sigma,m,\norm{k}_{L^\infty(\R)})) \\
    &\le -2^{2j}  w_j(t)  + 2^{2j} ( C(\sigma,m,\norm{k}_{L^\infty(\R)})2^{2j} S_{j-1}^2+ C(\sigma,m)).
  \end{align*}
  Moreover, define \(u(t,x):= -2^{2j} x  + 2^{2j} ( C(\sigma,m,\norm{k}_{L^\infty(\R)}) 2^{2j} S_{j-1}^2+ C(\sigma,m,\norm{k}_{L^\infty(\R)}))\), \(\epsilon:= 2^{2j}\) and \(A:= C(\sigma,m,\norm{k}_{L^\infty(\R)})2^{2j} S_{j-1}^2+ C(\sigma,m,\norm{k}_{L^\infty(\R)})\). Then, \(u\) is locally Lipschitz continuous in \(x\) and \(v=e^{-\epsilon t} v_0 + A (1-e^{-\epsilon t})\) is a solution of the following ODE
  \begin{align*}
    \begin{cases}
    \frac{\d}{\d t} v(t) = u(t,v(t)) \\
    v(0) = v_0
    \end{cases}.
  \end{align*}
  Let us choose \( v_0 := \int_{\R} \rho_0^{2^j} \Id x \ge w(0)\). Then, we can apply the comparison principle to obtain
  \begin{equation*}
    w_j(t) \le v(t) \le v_0 + A
    \le \norm{\rho_0}_{L^\infty(\R)}^{2^j-1} +  C(\sigma,m, \norm{k}_{L^\infty(\R)})2^{2j} S_{j-1}^2+  C(\sigma,m,\norm{k}_{L^\infty(\R)}).
  \end{equation*}
  It follows that
  \begin{equation*}
    S_j= \sup\limits_{t \in [0,T)} w_j(t) \le C(\sigma,m, \norm{k}_{L^\infty(\R)}) \max(\norm{\rho_0}_{L^\infty(\R)}^{2^j-1}, \, 2^{2j} S_{j-1}^2+ 1).
  \end{equation*}
  To complete the proof, we perform a version of Moser iteration technique to bound the \(L^\infty\)-norm. For \(\tilde{S_j}: = \frac{S_{j}}{\norm{\rho_0}_{L^\infty(\R)}^{2^j-1}} \) the last inequality provides us with
  \begin{equation*}
    \tilde{S_j} \le C(\sigma,m, \norm{k}_{L^\infty(\R)}) \max(1, 2^{2j} \tilde{S}_{j-1}^2 ).
  \end{equation*}
  Adding on both sides \(\delta > 0\) and taking the logarithm, we arrive at
  \begin{align*}
    \log(\tilde{S_j}+\delta)
    &\le \max(\log( C(\sigma,m, \norm{k}_{L^\infty(\R)})+\delta), \log( C(\sigma,m, \norm{k}_{L^\infty(\R)}) 2^{2j} \tilde{S}_{j-1}^2 +\delta))\\
    &\le 2 \log(\tilde{S}_{j-1}+\delta)  + j\log(4) + \log( C(\sigma,m, \norm{k}_{L^\infty(\R)}))
  \end{align*}
  for some new constant \(C(\sigma,m, \norm{k}_{L^\infty(\R)})>0\). This implies
  \begin{equation*}
    2^{-j}\log(\tilde{S_j}+\delta) -2^{1-j}  \log(\tilde{S}_{j-1}+\delta)
    \le  2^{-j} j\log(4)+ 2^{-j} C(\sigma,m, \norm{k}_{L^\infty(\R)})
  \end{equation*}
  for \(j \in \N\), where we used Lemma~\ref{lemma: lp_bound_hk} to not subtract infinity, i.e. \(\log(\tilde{S}_{j-1}+\delta) < \infty\).  Adding the above inequality over \(j=1,\ldots,J\), we find
  \begin{align*}
    2^{-J}\log(\tilde{S_J}+\delta) - \log(\tilde{S}_{0}+\delta)
    &= \sum\limits_{j=1}^J 2^{-j}\log(\tilde{S_j}+\delta) -2^{-(j-1)}  \log(\tilde{S}_{j-1}+\delta) \\
    &\le  \sum\limits_{j=1}^\infty  2^{-j} j\log(4)+ 2^{-j}  C(\sigma,m, \norm{k}_{L^\infty(\R)})  \\
    & \le C
  \end{align*}
  for a constant $C$ independent of \(J\) and \(\delta>0\). A straightforward way to see that the series is absolutely convergent is to apply the ratio criterion from elementary analysis.

  Now, we have \(\tilde{S}_0  = \sup\limits_{t \in [0,T)} \norm{\rho(\cdot,t)}_{L^1(\R)} = 1\) by mass conservation. Therefore, taking the exponential function on both sides and letting \(\delta \to 0 \), we discover
  \begin{equation*}
    S_{J}^{2^{-J}} \le C \norm{\rho_0}_{L^\infty(\R)}^{(2^J-1)2^{-J}}
    \le C(\rho_0) < \infty.
  \end{equation*}
  On the other hand, we have
  \begin{equation*}
    S_{J}^{2^{-J}} = \left( \sup\limits_{t \in [0,T)} \int_{\R} \rho_m^{2^J}(t,x) \Id x \right)^{\frac{1}{2^J}}
    = \sup\limits_{t \in [0,T)} \left( \int_{\R} \rho_m^{2^J}(t,x) \Id x  \right)^{\frac{1}{2^J}} .
  \end{equation*}
  Finally, we can take the limit \(J \to \infty\) to conclude
  \begin{align*}
    \sup\limits_{t \in [0,T)} \norm{\rho_m(t,\cdot)}_{L^\infty(\R)}
    &=\sup\limits_{t \in [0,T)} \lim\limits_{J\to \infty} \norm{\rho_m(t,\cdot)}_{L^{2^J}(\R)}
    \le  \limsup\limits_{J\to \infty} \sup\limits_{t \in [0,T)}  \norm{\rho_m(t,\cdot)}_{L^{2^J}(\R)} \\
    &= \lim\limits_{J\to \infty} S_{J}^{2^{-J}}
    \le C(\norm{k}_{L^\infty(\R)},\rho_0).
  \end{align*}
\end{proof}

%
%

\section{Local Lipschitz bound for the interaction force kernels}\label{sec: local_lipschitz_bound}

In this section we introduce a uniform Lipschitz assumption on the approximation sequence \((k^\epsilon, \epsilon>0)\) and show that most bounded confidence models, as used in the theory of opinion formation~\cite{noorazar2020recent}, satisfy this assumption.

At first glance, we notice that even though the interaction force kernels \(k^\epsilon\) is uniformly bounded it is not uniformly Lipschitz continuous in \(\epsilon\). Hence, the classical theory regarding Lipschitz continuous interaction force kernels on mean-field limits cannot be applied directly to the particles systems introduced in Subsection~\ref{subsec: particle system}. Instead we need use the properties of the convolution to derive uniform Lipschitz continuity of the mean-field force \(k^\epsilon * \rho^\epsilon\). Following e.g. \cite{canizares2017stochastic, lazarovici2017mean, Fetecau_2019}, we derive a Lipschitz bound for certain models in the case where the trajectories \(X_t^\epsilon\) and \(Y_t^\epsilon\) are close in a suitable sense. This approach requires an approximation with suitable properties and could not be generalized, so far, to arbitrary approximations. The main reason lies in the derivative of the approximation \(k^\epsilon\). If the derivative would be non-negative, then we could use a Taylor approximation, the properties of the solution \(\rho^\epsilon\) and the formula \(\norm{\frac{\d }{\d x} k^\epsilon*\rho^\epsilon_t }_{L^\infty(\R)} = \norm{ k^\epsilon*\frac{\d }{\d x} \rho_t^\epsilon}_{L^\infty(\R)}\) to obtain a local Lipschitz bound for \(k^\epsilon\) with estimates on the gradient \(\frac{\d}{\d x} \rho_t^\epsilon\). Unfortunately, in most cases a simple mollification of \(k\) has a derivative becoming non-negative as well as non-positive. Therefore, we have to postulate the following assumptions on the approximation sequence \((k^\epsilon, \epsilon> 0)\).

\begin{assumption}\label{ass: loc_lip_bound}
  The sequence \((k^\epsilon, \epsilon>0)\) satisfies the following:
  \begin{enumerate}[label=(\roman*)]
    \item There exists a family of functions \((l^\epsilon, \epsilon>0)\) such that
          \begin{equation*}
            |k^\epsilon(x)-k^\epsilon(y)|\le l^\epsilon(y) |x-y|
          \end{equation*}
          for \(x,y \in \R\) with \(|x-y| \le 2 \epsilon^{-1}\);
    \item \begin{equation*}
             \sup\limits_{t \in [0,T]} \norm{l^\epsilon*\rho^\epsilon_t}_{L^\infty(\R)} \le  C(\norm{k}_{L^\infty(\R)}) (\norm{\rho_0}_{L^1(\R)} + \norm{\rho_0}_{L^\infty(\R)}) ,
          \end{equation*}
          where \(  C(\norm{k}_{L^\infty(\R)})\) is some finite constant depending on the \(L^\infty(\R)\)-norm of \(k\).
   \end{enumerate}
\end{assumption}  

\begin{remark}
  The constant \(2\) in Assumption~\ref{ass: loc_lip_bound} can be replaced by any positive constant. For simplicity, we choose the most convenient one to avoid cumbersome notation.
\end{remark}


\subsection{Exemplary interaction force kernels}

The particle systems, as introduced in Subsection~\ref{subsec: particle system}, can be used to model the opinion of interacting individuals, see e.g. \cite{noorazar2020recent}. A prominent class are given by so-called bounded confidence models, in which the interaction is described by interaction force kernels of the form
\begin{equation*}
  k_{\scriptscriptstyle{BCM}}(x):= \indicator{[0,R]}(|x|) h(x),\quad \text{with} \quad h \in C^2(\R).
\end{equation*}
To show that \( k_{\scriptscriptstyle{BCM}}\) satisfies Assumption~\ref{ass: loc_lip_bound}, we introduce the following approximation sequence \((\psi_{a,b}^\epsilon, \epsilon>0)\) of the indicator function \(\indicator{[a,b]}(x)\) with \(a,b\in \R\), \(a < b\), such that the following properties hold for each \(\epsilon > 0\):
\begin{itemize}
  \item \(\psi_{a,b}^\epsilon \in \testfunctions{\R}\),
  \item \(\psi_{a,b}^\epsilon \to \indicator{[a,b]}\)   as \(\epsilon \to 0\) almost everywhere,
  \item \(\mathrm{supp}(\psi_{a,b}^\epsilon) \subseteq [a-2 \epsilon , b + 2 \epsilon] \), \(\mathrm{supp}(\frac{\d}{\d x} \psi_{a,b}^\epsilon ) \subset [a-2\epsilon, a+2 \epsilon] \cup [b -2 \epsilon , b +2\epsilon] \),
  \item \(0 \le \psi_{a,b}^\epsilon \le 1 \), \(|\frac{\d}{\d x} \psi_{a,b}^\epsilon | \le \frac{C}{\epsilon} \) for some constant \(C>0 \).
\end{itemize}
Since we want to take \(\epsilon \to 0\), we consider only the case where \(\epsilon\) is small enough. In particular, we can take the mollification of the indicator function of a set. We define the regularized interaction force kernel
\begin{equation*}
  k_{\scriptscriptstyle{BCM}}^\epsilon(x) = \psi_{-R,R}^\epsilon (x) h(x) \in C_c^2(\R),
\end{equation*}
which obviously satisfies Assumptions~\ref{ass: kernel_smooth_convergence}. That it also satisfies Assumption~\ref{ass: loc_lip_bound} is verified in the following.

\begin{lemma}[Local Lipschitz bound for bounded confidence models]\label{lemma: local_lipschitz_bound_of_K_N}
  Consider the regularized interaction force kernel \(k_{\scriptscriptstyle{BCM}}^\epsilon\) with cut-off \(\epsilon\). Moreover, let \(\mathrm{NBR}_\epsilon := [-R-4\epsilon, -R+ 4\epsilon] \cup [R-4\epsilon, R+ 4\epsilon]\) (``neighborhood of \(R\)''). Then, we have the following estimates:
  \begin{enumerate}[label = (\roman*)]
    \item For each \(x,y \in \R\) with \(|x-y| \le 2 \epsilon\) and
    \begin{align*}
      l^\epsilon_{\scriptscriptstyle{BCM}}(y) : = \begin{cases}
				 C \indicator{[-R-3,R+3]}(y), \quad &  y \in \mathrm{NBR}_\epsilon^{\mathrm{c}}\\
				 C \epsilon^{-1},  \quad & y \in \mathrm{NBR}_\epsilon
				\end{cases}
    \end{align*}
    it holds that
    \begin{equation*}
      |k_{\scriptscriptstyle{BCM}}^\epsilon(x)-k_{\scriptscriptstyle{BCM}}^\epsilon(y)| \le l^\epsilon_{\scriptscriptstyle{BCM}}(y) | x-y|;
    \end{equation*}
    \item For each  \(x,y \in \R^{N}\) with \(|x-y|_{\infty} \le \epsilon\) and
    \begin{equation*}
      L_{i,\scriptscriptstyle{BCM}}^\epsilon(y_1,\dots,y_N):= \frac{1}{N}\sum\limits_{j=1}^N l^\epsilon_{\scriptscriptstyle{BCM}}(y_i-y_j), \quad (y_1, \dots, y_N) \in \R^{N},
    \end{equation*}
    it holds that
    \begin{equation*}
      |K_{i,\scriptscriptstyle{BCM}}^\epsilon(x)-K_{i,\scriptscriptstyle{BCM}}^\epsilon(y)| \le 2L_{i,\scriptscriptstyle{BCM}}^\epsilon (y) |x-y|_{\infty},
    \end{equation*}
    where \(K_{i,\scriptscriptstyle{BCM}}\) is defined by~\eqref{eq: K_N} with \(k_{\scriptscriptstyle{BCM}}\). 
  \end{enumerate}
\end{lemma}

\begin{proof}
  (\romannumeral 1) Let \(|x-y| \le 2 \epsilon^{-1}\). By the mean value theorem, we have the bound
  \begin{equation*}
    |k^\epsilon_{\scriptscriptstyle{BCM}}(x)-k^\epsilon_{\scriptscriptstyle{BCM}}(y)| \le \bigg|\frac{\d}{\d x} k^\epsilon (z)\bigg| | x-y|
  \end{equation*}
  for some \(z\) in the line segment between \(x\) and \(y\). Let us distinguish between two cases.

  \textit{Case} \(1\): \(y \in \mathrm{NBR}_\epsilon \). Using the bound \[\bigg|\frac{\d}{\d x} k^\epsilon_{\scriptscriptstyle{BCM}}(z)\bigg| \le \bigg| \frac{\d}{\d x} \psi^\epsilon(z) h(z) \bigg| + \bigg| \psi^\epsilon(z) \frac{\d}{\d x} h(z) \bigg| \le C \epsilon^{-1} \] for all \(z \in \R\) for some constant \(C>0\), which depends on the deterministic function \(h\), it follows
  \begin{equation*}
    |k^\epsilon_{\scriptscriptstyle{BCM}}(x)-k^\epsilon_{\scriptscriptstyle{BCM}}(y)| \le  C \epsilon^{-1} | x-y|. 
  \end{equation*}

  \textit{Case} \(2\): \(y \in \mathrm{NBR}_N^{\mathrm{c}}\). Because \(z\) lies on the line segment between \(x,y\), it follows for some \(s \in [0,1]\) that
  \begin{equation*}
    |z-y|=|y-s(x-y)-y| \le |x-y| \le 2 \epsilon
  \end{equation*}
  and therefore \(|R-z| \ge |R-y|-|z-y| \ge 4 \epsilon -  2  \epsilon =2  \epsilon \). Analogously, \(|-R-z|  \ge |-R-y|-|z-y| \ge 2  \epsilon\). Consequently, \(z\) is far enough away from the points \(R\) and \(-R\) such that the derivative of the approximation \(\frac{\d}{\d x} \psi^{\epsilon}\) vanishes. This implies
  \begin{equation*}
    \bigg|\frac{\d}{\d x} k^\epsilon_{\scriptscriptstyle{BCM}}(z)\bigg| \le \bigg| \psi^\epsilon(z) \frac{\d}{\d x} h(z) \bigg|
    \le \bigg|  \frac{\d}{\d x} h(x) \bigg| \indicator{[-R-3,R+3]}(y)
    \le C \indicator{[-R-3,R+3]}(y),
  \end{equation*}
  where we used \(|y| \le |y-z|+|z| \le 2 +|z|\). Together with the mean value theorem this proves the second case.

  (\romannumeral 2) We want to apply (\romannumeral 1). For \(x,y \in \R^{N}, \, |x-y|_\infty \le \epsilon\), it follows
  \begin{align*}
    |K^\epsilon_{i,\scriptscriptstyle{BCM}}(x)-K^\epsilon_{i,\scriptscriptstyle{BCM}}(y)|
    &\le \frac{1}{N-1} \sum\limits_{\substack{j=1 \\ j \neq i}}^N |k^\epsilon_{\scriptscriptstyle{BCM}}(x_i-x_j)-k^\epsilon_{\scriptscriptstyle{BCM}}(y_i-y_j)|\\
    & \le \frac{1}{N-1} \sum\limits_{\substack{j=1 \\ j \neq i}}^N l^\epsilon_{\scriptscriptstyle{BCM}} (y_i-y_j) |x_i-x_j-(y_i-y_j)| \\
    &\le 2 L^\epsilon_{i,\scriptscriptstyle{BCM}}(y) |x-y|_\infty.
  \end{align*}
  It is indeed justified to apply (\romannumeral 1) since \(|x_i-x_j-(y_i-y_j)| \le 2 |x-y|_\infty \le 2 \epsilon\) for all \(i,j=1,\dots,N\).
\end{proof}

\begin{remark}\label{rem: definition_of_big_L}
  The second part of Lemma~\ref{lemma: local_lipschitz_bound_of_K_N} is a direct consequence of part one. Hence, if \((k^\epsilon, \epsilon> 0)\) satisfies Assumption~\ref{ass: loc_lip_bound}, we have
  \begin{equation*}
    |K_i^\epsilon(x)-K_i^\epsilon(y)| \le 2L_i^\epsilon (y) |x-y|_{\infty}
  \end{equation*}
  for \(x,y \in \R^N\) with \(|x-y|_\infty \le \epsilon\) and
  \begin{equation*}
    L_i^\epsilon(y_1,\dots,y_N):= \frac{1}{N}\sum\limits_{j=1}^N l^\epsilon(y_i-y_j), \quad (y_1, \dots, y_N) \in \R^{N}.
  \end{equation*}
\end{remark}

The convenient properties of the solutions \((\rho^\epsilon, \epsilon \ge 0 )\) allow us to find a uniform bound of the convolution term \(l^\epsilon * \rho_t^\epsilon\). This will be the content of the following lemma.

\begin{lemma}\label{lemma: l_infinity_bound_on_averaged_lipschitz_bound}
  Suppose Assumption~\ref{ass: initial condition} and let us define
  \begin{equation*}
    \bar{L}_{t,i,\scriptscriptstyle{BCM}}^\epsilon(y_1, \ldots, y_N) := (l^\epsilon_{\scriptscriptstyle{BCM}}*\rho_t^\epsilon)(y_i),  \quad (y_1, \dots, y_N) \in \R^{N}, 
  \end{equation*}
  the averaged version of \(L^\epsilon\) for \(i=1, \ldots, N\). Then, there exists a constant \(C \), depending on the deterministic function \(h\), such that
  \begin{equation*}
    \sup\limits_{i = 1, \ldots, N } \sup\limits_{t \in [0,T]} \|\bar{L}_{t,i,\scriptscriptstyle{BCM}}^\epsilon \|_{L^\infty(\R^N)} \le C \;  ( \norm{\rho_0^\epsilon}_{L^1(\R)} + \norm{\rho_0^\epsilon}_{L^\infty(\R)} ),
  \end{equation*}
  where \(\rho_t^\epsilon\) is the solution of~\eqref{eq: regularized_aggregation_diffusion_pde} for the special interaction force kernel \(k_{\scriptscriptstyle{BCM}}\).
\end{lemma}

\begin{proof}
  Let \(i \in \{1,\ldots, N\}\) and \(y=(y_1, \ldots, y_N) \in \R^N \). Then, by mass conservation and Lemma~\ref{lemma: infinity_bound_solution}, we have
  \begin{align*}
    |\bar{L}_{t,i,\scriptscriptstyle{BCM}}^\epsilon(y) |
    &= |(l^\epsilon_{\scriptscriptstyle{BCM}}*\rho_t^\epsilon)(y_i)| \\
    &\le  \int_{\R } \indicator{\{ z \,  : \, y_i-z \in \mathrm{NBR}_\epsilon^{\mathrm{c}} \}} |l^\epsilon_{\scriptscriptstyle{BCM}}(y_i-z) \rho_t(z)| \Id  z +
    \int_{\R } \indicator{\{z \,  : \, y_i-z \in \mathrm{NBR}_\epsilon\}} |l^\epsilon_{\scriptscriptstyle{BCM}}(y-z) \rho_t(z)| \Id  z \\
    &\le C \int_{\R} |\rho_t^\epsilon(z)| \Id  z  + C \epsilon  \epsilon^{-1}  \norm{\rho_t^\epsilon}_{L^\infty(\R)} \\
    &\le C (( \norm{\rho_t^\epsilon}_{L^1(\R)} +   \norm{\rho_t^\epsilon}_{L^\infty(\R)} ) \\
    &\le C (( \norm{\rho_0}_{L^1(\R)} +   \norm{\rho_0}_{L^\infty(\R)} ),
  \end{align*}
  where \(C\) again depends on the deterministic function \(h\). 
\end{proof}

Consequently, we have shown that \(k^\epsilon = \psi^\epsilon h\) fulfills Assumption~\ref{ass: loc_lip_bound} and (as previously mentioned) Assumption~\ref{ass: kernel_smooth_convergence}, which implies the following corollary.

\begin{corollary}
  The interaction force kernel \(k_{\scriptscriptstyle{BCM}}\) satisfies Assumptions~\ref{ass: kernel_smooth_convergence} and~\ref{ass: loc_lip_bound} with the associated approximation sequence \((k^\epsilon_{\scriptscriptstyle{BCM}}, \epsilon > 0)\) given by \(k^\epsilon_{\scriptscriptstyle{BCM}}(x) :=  \psi^\epsilon_{-R,R}(x)h(x)\).
\end{corollary}
 
Another example of interest are interaction forces with \(h(x) := \mathrm{sgn}(x)\), which corresponds to a uniform interaction, i.e., every particle in the interaction radius has the same impact. Unfortunately, \(\mathrm{sgn}(x) \notin C^2(\R)\) and, hence, we cannot directly apply Lemma~\ref{lemma: local_lipschitz_bound_of_K_N}. However, the function \(\mathrm{sgn}(x)\) has no effect on the discontinuities \(-R\) and \(R\). Therefore, if we can control the function around zero, we can obtain an analog result to Lemma~\ref{lemma: local_lipschitz_bound_of_K_N}. Indeed, we define
\begin{equation*}
  k_{\scriptscriptstyle{U}}(x) := - \indicator{[-R,0]}(x) + \indicator{[0,R]}(x), \quad x\in \R,
\end{equation*}
which can be appropriated by \(k^\epsilon_{\scriptscriptstyle{U}}(x) := \psi^\epsilon_{-R,0}(x) + \psi^\epsilon_{0,R}(x)\). Defining
\begin{equation*}
  \mathrm{NBZR}_\epsilon: = [-R-4\epsilon, -R+ 4\epsilon]\cup [-4\epsilon,4\epsilon] \cup [R-4\epsilon, R+ 4\epsilon]
\end{equation*}
as the neighbourhood of zero and \(R\), we can perform the same steps as in Lemma~\ref{lemma: local_lipschitz_bound_of_K_N} to prove the following Lemma for \( k_{\scriptscriptstyle{U}}\).

\begin{lemma}[Local Lipschitz bound for uniform kernel]\label{lemma: local_lipschitz_bound_of_uniform_K_N}
  Consider the regularized interaction force kernel \(k_{\scriptscriptstyle{U}}^\epsilon\) with cut-off \(\epsilon\). Then, we have the following estimates:
  \begin{enumerate}[label = (\roman*)]
    \item For each \(x,y \in \R\) with \(|x-y| \le 2 \epsilon\) and
    \begin{align*}
      l^\epsilon_{\scriptscriptstyle{U}}(y) : = \begin{cases}
				 0, \quad &  y \in \mathrm{NBZR}_\epsilon^{\mathrm{c}}\\
				 C \epsilon^{-1},  \quad & y \in \mathrm{NBZR}_\epsilon
				\end{cases}
    \end{align*}
    it holds that
    \begin{equation*}
      |k_{\scriptscriptstyle{U}}^\epsilon(x)-k_{\scriptscriptstyle{U}}^\epsilon(y)| \le l^\epsilon_{\scriptscriptstyle{U}}(y) | x-y|;
    \end{equation*}
    \item For each  \(x,y \in \R^{N}\) with \(|x-y|_{\infty} \le \epsilon\) and
    \begin{equation*}
      L_{i,\scriptscriptstyle{U}}^\epsilon(y_1,\dots,y_N):= \frac{1}{N}\sum\limits_{j=1}^N l^\epsilon_{\scriptscriptstyle{U}}(y_i-y_j), \quad (y_1, \dots, y_N) \in \R^{N},
    \end{equation*}
    it holds that
    \begin{equation*}
      |K_{i,\scriptscriptstyle{U}}^\epsilon(x)-K_{i,\scriptscriptstyle{U}}^\epsilon(y)| \le 2L_{i,\scriptscriptstyle{U}}^\epsilon (y) |x-y|_{\infty}.
    \end{equation*}
  \end{enumerate}
\end{lemma}

\begin{corollary}
  The interaction force kernel \(k_{\scriptscriptstyle{U}}\) satisfies Assumptions~\ref{ass: kernel_smooth_convergence} and~\ref{ass: loc_lip_bound} with the associated approximation sequence \((k^\epsilon_{\scriptscriptstyle{U}}, \epsilon > 0)\) given by \(k^\epsilon_{\scriptscriptstyle{U}}(x) :=  \psi^\epsilon_{-R,0}(x) + \psi^\epsilon_{0,R}(x)\).
\end{corollary}

\begin{proof}
  Apply Lemma~\ref{lemma: local_lipschitz_bound_of_uniform_K_N} and similar computations as in proof of Lemma~\ref{lemma: l_infinity_bound_on_averaged_lipschitz_bound} to show that Assumption~\ref{ass: loc_lip_bound} is fulfilled. The verification of Assumption~\ref{ass: kernel_smooth_convergence} follows immediately.
\end{proof}

\section{Law of large numbers}\label{sec: law_of_large_numbers}

The derivation of propagation of chaos is based on defining several exceptional sets where the desired properties will not hold. Hence, we need to rely on the fact that the probability measure of these sets is extremely small. This fact is the subject of the next proposition.

\begin{proposition}[Law of large numbers]\label{prop: law_of_large_numbers}
  Let \(0< \alpha, \delta \) such that \( 0 < \delta+\alpha < 1/2\) and \(Z^1, \ldots, Z^N\) be independent random variables in \(\R\) such that \(Z^{i}\) has density \(u^{i}\) for \(i =1, \ldots, N \). Let \(h:\R \to \R \) be a bounded measurable function. Define \(H_i(Z):= \frac{1}{N} \sum\limits_{\substack{j=1 \\ j \neq i }}^N h(Z^{i}-Z^{j})\) and
  \begin{align*}
    &S:= \left \{ \sup\limits_{1\le i \le N } |H_i(Z)-\E(H_i(Z))| \ge N^{-(\delta+\alpha)} \right\}, \\
    &\widetilde{S}:= \left \{ \sup\limits_{1\le i \le N } |H_i(Z)-\E_{(-i)}(H_i(Z))| \ge N^{-(\delta+\alpha)} \right\},
  \end{align*}
  where \(\E_{(-i)}\) stands for the expectation with respect to every variable except \(Z^{i}\), i.e.
  \begin{equation*}
    \E_{(-i)}(H_i(Z)):= \frac{1}{N} \sum\limits_{\substack{j=1 \\ j \neq i }}^N (h * u^j ) (Z_i).
  \end{equation*}
  Then, for each \(\gamma >0\) there exists a constant \(C(\gamma) >0\), which depends on \(\gamma, C \), such that
  \begin{equation*}
    \P(S), \, \P(\widetilde{S}) \le C(\gamma) N^{-\gamma}.
  \end{equation*}
\end{proposition}

\begin{proof}
  We prove the statement for the set \(S\). The estimate for the set \(\widetilde{S}\) can be shown similarly by replacing \(\E(H_i(Z))\) with \(\E_{(-i)}(H_i(Z))\). First, we notice
  \begin{equation*}
    \P \left(\sup\limits_{1\le i \le N } |H_i(Z)-\E(H_i(Z))| \ge N^{-(\delta+\alpha)} \right) \le \sum\limits_{i=1}^N 
    \P(|H_i(Z)- \E(H_i(Z))| \ge N^{-(\delta+\alpha)}).
  \end{equation*}
  Hence, it suffices to prove
  \begin{equation*}
    \P(|H_i(Z)-\E(H_i(Z))| \ge N^{-(\delta+\alpha)}) \le C(\gamma) N^{-\gamma}
  \end{equation*}
  for each \(\gamma >0\), \(i=1,\ldots,N\). Let us assume \(i=1\) and for \(j=2,\ldots,N\) let us denote by \(\Theta_j \) the independent random variables \(\Theta_j := h(Z^1-Z^j)\). Then, applying Chebyshev's inequality to the function \(x \mapsto x^{2m}\), we obtain
  \begin{align}\label{eq: law_of_large_numbers_chebyshev}
    \begin{split}
    \P(|H_1(Z)-\E(H_1(Z))| \ge N^{-(\delta+\alpha)})
    &\le N^{2(\delta+\alpha) m} \E(|H_1(Z)-\E(H_1(Z))|^{2m}) \\
    &\le N^{2(\delta+\alpha) m} \E\left( \left( \frac{1}{N-1}  \sum\limits_{j=2}^N (\Theta_j -\E(\Theta_j)) \right)^{2m} \right).
    \end{split}
  \end{align}
  The expectation on the right-hand side can be rewritten, using the multinomial formula, as
  \begin{equation*}
    (x_2+x_3+ \cdots + x_N)^{2m} = \sum\limits_{a_2+a_3+\cdots + a_N =2m} \binom{2m}{a_2, \ldots ,a_N} \prod\limits_{j=2}^N x_j^{a_j} ,
  \end{equation*}
  where \(a=(a_2,a_3,\ldots,a_N) \in \N_0^{N-1}\) is a multiindex of length \(|a|= 2m \). Consequently, using the independence of \((\Theta_j, j=2,\ldots, N)\), we get
  \begin{align}\label{eq: law_of_large_number_intermediate_step}
    &\E\left( \left( \frac{1}{N}  \sum\limits_{j=2}^N (\Theta_j -\E(\Theta_j)) \right)^{2m} \right) \nonumber\\
    &\quad= \,  N^{-2m} \sum\limits_{a_2+a_3+\cdots + a_N =2m} \binom{2m}{a_2, \ldots ,a_N} \prod\limits_{j=2}^N \E((\Theta_j-\E(\Theta_j))^{a_j})\nonumber\\
    &\quad= \,  N^{-2m} \sum\limits_{\substack{a_2+a_3+\cdots + a_N =2m \\ |a|_0 \le m }} \binom{2m}{a_2, \ldots ,a_N} \prod\limits_{j=2}^N  \E((\Theta_j-\E(\Theta_j))^{a_j}),
  \end{align}
  where \(|a_0|\) the number of non-zero entries of the multiindex \(a\). Otherwise, if \(|a|_0 > m\), then there exists a \(j\) such that \(a_j=1\) and the product vanish since \(\E(\Theta_j-\E(\Theta_j)) =0\). From the bound on \(h \) we have
  \begin{equation*}
    |\E((\Theta_j-\E(\Theta_j))^{a_j})|
    = \left| \int_{\R \times \R } (h(z_1-z_j)-E(\Theta_j))^{a_j} u^{1}(z_1) u^j(z_j) \Id z_1 \Id z_j \right|
    \le C^{a_j}.
  \end{equation*}
  Using the facts
  \begin{equation*}
    \binom{2m}{a_2, \ldots ,a_N} \le (2m)^{2m} \quad \mathrm{and} \quad
    \sum\limits_{\substack{a_2+a_3+\cdots + a_N =2m \\ |a|_0 = k }} 1 \le N^k (2m)^k
  \end{equation*}
  for \(0 \le k \le m \), we can estimate \eqref{eq: law_of_large_number_intermediate_step} to arrive at
  \begin{align*}
    \E\left( \left( \frac{1}{N}  \sum\limits_{j=2}^N (\Theta_j -\E(\Theta_j)) \right)^{2m} \right)
    &\le N^{-2m} \sum\limits_{\substack{a_2+a_3+\cdots + a_N =2m \\ |a|_0 \le m }}   (2m)^{2m} C^{2m}  \\
    &\le  N^{-2m} \sum\limits_{k=1}^{m} N^k (2m)^{3m} C^{2m}
    \le \frac{C(m) N^m }{N^{2m}}
  \end{align*}
  for some constant \(C(m)\). Hence, plugging it into \eqref{eq: law_of_large_numbers_chebyshev}, we find
  \begin{equation*}
    \P(|H_1(Z)-E(H_1(Z))| \ge N^{-(\delta+\alpha)})
    \le C(m)  \frac{N^{2(\delta+\alpha) m +m}}{N^{2m}}.
  \end{equation*}
  Using the assumption \(\delta+\alpha < 1/2\) and choosing \(m\) such that \(m(-1+2 (\delta+\alpha)) = \gamma\) proves the proposition.
\end{proof}

The law of large numbers provided in Proposition~\ref{prop: law_of_large_numbers} allows to show that the sets, where the desired properties do not hold, are small in probability.

\begin{corollary}\label{cor: law_of_large_numbers}
  Let \(0< \alpha, \delta\), \(0<\alpha+\delta<1/2\), \(\epsilon \sim N^{-\beta}\) with \(0< \beta \le \alpha\) and define for \(0 \le t \le T\) the following sets
  \begin{align*}
    &B_t^1 := \{|K^\epsilon(\mathbf{Y}_t^{N,\epsilon})-\overline{K^\epsilon_{t}}(\mathbf{Y}_t^{N,\epsilon})|_\infty \le N^{-(\delta+\alpha)}  \}, \\
    & B_t^2 := \{|L^\epsilon(\mathbf{Y}_t^{N,\epsilon}) - \overline{L_{t}^\epsilon}(
    \mathbf{Y}_t^{N,\epsilon}) |_\infty \le 1\}, 
  \end{align*}
  where the mean-field particles are close under the kernel \(K^\epsilon\) and \(L^\epsilon\), which were defined in Section~\ref{subsec: particle system} and Remark~\ref{rem: definition_of_big_L}. Then, for each \(\gamma > 0\) there exists a \(C(\gamma) > 0\) such that
  \begin{equation*}
    \P( (B_t^1)^{\mathrm{c}}) ,\P((B_t^2)^{\mathrm{c}})
    \le C(\gamma) N^{-\gamma }
  \end{equation*}
  for every \(0 \le t \le T\), where the constant \(C(\gamma) \) is independent of \(t \in [0,T]\).
\end{corollary}

\begin{proof}
  First, the random variables \((Y_t^{i,\epsilon}, i =1,\ldots, N )\) are i.i.d. and have a probability density~\(\rho_t^\epsilon\) given by the solution of the regularized system~\eqref{eq: regularized_aggregation_diffusion_pde}. Moreover, we have
  \begin{equation*}
    K_i^\epsilon(x_1,\ldots,x_N)= -\frac{1}{N} \sum\limits_{j=1}^N k^\epsilon(x_i-x_j), \quad (x_1,\ldots,x_N) \in \R^{N},
  \end{equation*}
  with \(k^\epsilon\) bounded. We recall that we denote by \(\E_{(-i)}\) the expectation with respect to every variable but the \(i\)-th. Therefore, we get
  \begin{align*}
    \E_{(-i)}(K_i^\epsilon(\mathbf{Y}_t^{N,\epsilon}))
    &= -\frac{1}{N} \sum\limits_{j=1}^N \E( k^\epsilon(Y_t^{i,\epsilon}-Y_t^{j,\epsilon})) \\
    &=  -\frac{1}{N} \sum\limits_{j=1}^N
    \int_{\R}  k^\epsilon(Y_t^{i,\epsilon}-z)  \rho^\epsilon(z,t) \Id z
    =- (k^\epsilon*\rho_t)(Y_t^{i,\epsilon})
    =\overline{K^{\epsilon}_{t,i}}(\mathbf{Y}_t^{N,\epsilon})
  \end{align*}
  for all \(i=1,\ldots,N\). As a result, we obtain
  \begin{equation*}
    (B_t^1)^{\mathrm{c}} = \left\{\sup\limits_{1 \le i \le N} |K_{i}^\epsilon(\mathbf{Y}_t^{N,\epsilon})-\E_{(-i)}(K_i^\epsilon(\mathbf{Y}_t^{N,\epsilon}))| > N^{-(\delta+\alpha)} \right\}
  \end{equation*}
  and therefore, by Proposition~\ref{prop: law_of_large_numbers},
  \begin{equation*}
    \P( (B_t^1)^{\mathrm{c}}) \le C(\gamma) N^{-\gamma }.
  \end{equation*}

  For the set \(B_t^2\) we notice the function \(l^\epsilon N^{-\alpha}\) is bounded since \(\epsilon \sim N^{-\beta}\) and, thus, we can do similar steps as before with the set
  \begin{align*}
    (B_t^1)^{\mathrm{c}}
    &= \{N^{-\alpha}|L^\epsilon(\mathbf{Y}_t^{N,\epsilon}) - \overline{L_{t}^\epsilon}(\mathbf{Y}_t^{N,\epsilon}) |_\infty \ge N^{-\alpha}\}  \\
    &\subseteq \{N^{-\alpha}|L^\epsilon(\mathbf{Y}_t^{N,\epsilon}) - \overline{L_{t}^\epsilon}(\mathbf{Y}_t^{N,\epsilon}) |_\infty \ge N^{-(\delta + \alpha)}\}.
  \end{align*}
  This proves the corollary.
\end{proof}

\section{Propagation of chaos in probability}\label{sec: convergence_in_probability}

In this section we are going to prove propagation of chaos for the particle system~\eqref{eq: regularized_particle_system}. We deploy a coupling method with the mean-field SDE~\eqref{eq: mean_field_trajectories} and show convergence in probability with an arbitrary algebraic rate \(N^{-\gamma}\) for \(\gamma > 0\). To that end, we present the main result, which states that the trajectory of the N-particle system \(X^N\) with \(X^N_0 \sim \overset{N}{\underset{i=1}{\otimes}} \rho_0 \) typically remains close to the mean-field trajectory \(Y^N\) with same starting position \(X^N_0=Y_0^N\) during any finite interval \([0,T]\).

\begin{theorem}\label{theorem: main_theorem}
  Suppose Assumption~\ref{ass: initial condition}. Let \(T>0\), \(\alpha \in \left(0,\frac{1}{2}\right)\) and \((k^\epsilon, \epsilon > 0)\) satisfy Assumptions~\ref{ass: kernel_smooth_convergence} and~\ref{ass: loc_lip_bound} with \(\epsilon \sim N^{-\beta}\) for \(0 < \beta \le \alpha\). Then, for every \(\gamma >0\), there exists a positive constant \(C(\gamma)\) and \(N_0 \in \N\) such that
  \begin{equation*}
    \P\left(\sup\limits_{t \in [0,T]} |\mathbf{X}_t^{N,\epsilon}-\mathbf{Y}_t^{N,\epsilon}|_\infty \ge N^{-\alpha}\right) \le C(\gamma) N^{-\gamma}, \quad \mathrm{for \; each}  \; N \ge N_0.
  \end{equation*}
  The constant \(C(\gamma)\) depends on the initial density \(\rho_0\), the final time \(T>0\), \(\alpha\) and \(\gamma\). The natural number \(N_0\) also depends on \(\rho_0\), \(T\) and \(\alpha\).
\end{theorem}

To prove Theorem~\ref{theorem: main_theorem}, we need the following auxiliary lemma.

\begin{lemma}\label{lemma: sup_dini_derivative_bound}\textup{\cite[Lemma~8.1]{lazarovici2017mean}}
  For a function \(f\colon\R \to \R\) we denote the right upper Dini derivative by
  \begin{equation*}
    \overline{D}_t^+f(y) := \limsup\limits_{h \to 0^+} \frac{f(y+h)-f(x)}{h}.
  \end{equation*}
  Let \(g \in C^1(\R)\) and \(h(y) := \sup\limits_{0 \le s \le y } g(s)\). Then, one has \(\overline{D}_t^+ h(y) \leq \max\left( 0, \frac{\d}{\d t} g(y)  \right)\) for all \(y\geq 0\).
\end{lemma}

\begin{proof}[Proof of Theorem~\ref{theorem: main_theorem}]
  For \(T>0\) and \(\alpha \in (0,1/2)\) and \(\delta = \frac{1}{2} (1/2 - \alpha) > 0\) let us define the auxiliary process
  \begin{equation*}
    J_t^N := \min \left( 1, \sup\limits_{0 \le s \le t } e^{\lambda(T-s)} (N^\alpha
    |\mathbf{X}_s^{N,\epsilon} -\mathbf{Y}_s^{N,\epsilon}|_\infty + N^{-\delta}) \right),
  \end{equation*}
  where \(\lambda>0\) is a constant, which will be defined later. In the first step we want to understand how \(J_t^N\) helps us to control the maximum distance \(|\mathbf{X}_t^{N,\epsilon} -\mathbf{Y}_t^{N,\epsilon}|_\infty\). For \(0 \le t \le T \) we have
  \begin{align}\label{eq: sup_N_alpha_z_less_than_1}
    \begin{split}
    \sup\limits_{0 \le s \le t} N^\alpha |\mathbf{X}_s^{N,\epsilon} -\mathbf{Y}_s^{N,\epsilon}|_\infty
    & \le \sup\limits_{0 \le s \le t} e^{\lambda (T-s)} (N^\alpha
    |\mathbf{X}_s^{N,\epsilon} -\mathbf{Y}_s^{N,\epsilon}|_\infty + N^{-\delta}).
    \end{split}
  \end{align}
  Hence, if \(J_t^N < 1 \) we obtain \(\sup\limits_{0 \le s \le t} N^\alpha |\mathbf{X}_s^{N,\epsilon} -\mathbf{Y}_s^{N,\epsilon}|_\infty \le  J_t^N < 1\). Furthermore, we can assume \(N\ge N_0\) such that \( J_o^N= e^{\lambda (T-s)} N^{-\delta} < \frac{1}{2} \) with \(N_0\) depending on \(T, \lambda\). As a result, we find
  \begin{align}\label{eq: J_control}
    \P\left(\sup\limits_{ t \in [0,T]} |\mathbf{X}_t^{N, \epsilon}-\mathbf{Y}_t^{N,\epsilon}| \ge N^{-\alpha}\right)
    &\le \P(J_T^N \ge 1)  \le \P\left( J_T^N - J_0^N \ge  \frac{1}{2} \right)\nonumber \\
    &\le 2\E(J_T^N - J_0^N)
    \le 2 \E \left( \int\limits_0^T  \overline{D}^+_t J_t^N \Id t \right)\\
    & = 2  \int\limits_0^T  \E( \overline{D}^+_t J_t^N )  \Id t, \nonumber
  \end{align}
  where we used a more general fundamental theorem of calculus, see e.g. \cite[Theorem~11]{Hagooddini2006}, in the last inequality. In the next step we want to estimate the Dini derivative \(\overline{D}^+_t J_t^N\). Applying Lemma~\ref{lemma: sup_dini_derivative_bound}, we discover
  \begin{equation}\label{eq: dini_derivative_J_inequality}
    \overline{D}^+_t J_t^N \le \max \left( 0,\frac{\d}{\d t} g(t) \right)
  \end{equation}
  with \(g(t):=  e^{\lambda(T-t)} (N^\alpha |\mathbf{X}_t^{N,\epsilon} -\mathbf{Y}_t^{N,\epsilon}|_\infty + N^\delta)\). Using the fact \( \partial^+ \max(h,l) = \max(\partial^+ h , \partial^+ l)\) for right-differentiable functions \(h\) and \(l\), where \(\partial^+\) denotes the right-derivative, we find
  \begin{equation}\label{eq: representation_of_the_derivative_g}
    \frac{\d}{\d t} g(t)
    = -\lambda  e^{\lambda(T-t)} (N^\alpha |\mathbf{X}_t^{N, \epsilon} -\mathbf{Y}_t^{N,\epsilon}|_\infty + N^{-\delta})  +
    e^{\lambda(T-t)} N^\alpha |K^\epsilon(\mathbf{X}_t^{N, \epsilon}) - \overline{K}_t^\epsilon(\mathbf{Y}_t^{N,\epsilon}) |_\infty
  \end{equation}
  with \(K^\epsilon\) and \(\overline{K}_t^\epsilon\) defined as in \eqref{eq: K_N} and \eqref{eq: averaged_K_N}, respectively. Next, let us introduce the set \(A_t:= \{\overline{D}^+_t J_t^N >0\}\) and notice that \eqref{eq: dini_derivative_J_inequality} implies \(A_t \subseteq \{\overline{D}^+_t J_t^N  \le \frac{\d}{\d t} g(t) \}\). Hence, we discover
  \begin{equation*}
    \E( \overline{D}^+_t J_t^N )
    = \E(\overline{D}^+_t J_t^N \indicator{A_t}) + \E(\overline{D}^+_t J_t^N \indicator{A_t^{\mathrm{c}}} )
    \le \E \left(\frac{\d}{\d t} g(t)  \indicator{A_t} \right).
  \end{equation*}
  In combination with \eqref{eq: J_control} we see that, in order to prove the theorem, it is enough to show that \( \E(\frac{\d}{\d t} g(t) \indicator{A_t})\) is bounded by \(C(\gamma) N^{-\gamma}\) for some constant \(C(\gamma) > 0\) and \( t \in [0,T]\).

  At this moment let us recall the sets \(B_t^1,B_t^2\) from Section~\ref{sec: law_of_large_numbers}, where the ``good'' properties hold to further reduce the problem. We have
  \begin{align*}
    \E \left(\frac{\d}{\d t} g(t)  \indicator{A_t} \right)
    &= \E \left(\frac{\d}{\d t} g(t)  \indicator{A_t \cap B_t^1 \cap B_t^2} \right) + \E \left(\frac{\d}{\d t} g(t)  \indicator{A_t \cap (B_t^1 \cap B_t^2)^{\mathrm{c}}} \right)   \\
    &\le\E \left(\frac{\d}{\d t} g(t)  \indicator{A_t \cap B_t^1 \cap B_t^2} \right) + C N^\alpha (\P((B_t^1)^{\mathrm{c}}) + \P((B_t^2)^{\mathrm{c}})) \\
    &\le \E \left(\frac{\d}{\d t} g(t \indicator{A_t \cap B_t^1 \cap B_t^2} \right) + C(\gamma) N^{-\gamma},
  \end{align*}
  where we used the fact that the interaction force approximation \(k^\epsilon\) is uniformly bounded in the first inequality and thus we have \(|\frac{\d}{\d t} g(t)| \le C N^\alpha\) with the help of \eqref{eq: representation_of_the_derivative_g}. The last inequality follows immediately from Corollary~\ref{cor: law_of_large_numbers} and relabeling \(\gamma\). It is therefore enough to prove that \(\frac{\d}{\d t} g(t) \le 0 \) holds under the event \(A_t \cap B_t^1 \cap B_t^2\). This is equivalent to the inequality
  \begin{equation}\label{eq: equivalent_derivative_vanishing}
    e^{\lambda(T-t)} N^\alpha |K^\epsilon(\mathbf{X}_t^{N , \epsilon}) - \overline{K}_t^\epsilon(\mathbf{Y}_t^{N,\epsilon}) |_\infty
    \le \lambda  e^{\lambda(T-t)} (N^\alpha |\mathbf{X}_t^{N, \epsilon} -\mathbf{Y}_t^{N,\epsilon}|_\infty + N^{-\delta}) .
  \end{equation}
  We observe that on \(A_t\) we have \(J_t^N < 1\). Together with \eqref{eq: sup_N_alpha_z_less_than_1} this means
  \begin{equation}\label{eq: trajectories_close_under_a_t}
    \sup\limits_{0 \le s \le t} |\mathbf{X}_s^{N , \epsilon}-\mathbf{Y}_s^{N,\epsilon}|_\infty \le N^{-\alpha}
  \end{equation}
  holds on \(A_t\). Splitting up the term on the left-hand side of \eqref{eq: equivalent_derivative_vanishing}, we obtain
  \begin{align*}
    |K^\epsilon (\mathbf{X}_t^{N, \epsilon}) - \overline{K}_t^\epsilon(\mathbf{Y}_t^{N,\epsilon}) |_\infty
    &\le |K^\epsilon(\mathbf{X}_t^{N, \epsilon}) - K^\epsilon(\mathbf{Y}_t^{N,\epsilon})|_\infty +| K^\epsilon(\mathbf{Y}_t^{N,\epsilon})-\overline{K}_t^\epsilon(\mathbf{Y}_t^{N,\epsilon}) |_\infty \\
    &\le |L^\epsilon(\mathbf{Y}_t^{N,\epsilon})|_\infty |\mathbf{X}_t^{N, \epsilon }-\mathbf{Y}_t^{N,\epsilon}|_\infty + N^{-(\delta+\alpha)}  \\
    &\le (C+| \overline{L_{t}^\epsilon}(\mathbf{Y}_t^{N,\epsilon})|_\infty ) |\mathbf{X}_t^{N,\epsilon} -\mathbf{Y}_t^{N,\epsilon}|_\infty +N^{-(\delta+\alpha)}  \\
    &\le C(\rho_0,T) ( |\mathbf{X}_t^{N,\epsilon} -\mathbf{Y}_t^{N,\epsilon}|_\infty +N^{-(\delta+\alpha)} ),
  \end{align*}
  where we used the local Lipschitz bound from Assumption~\ref{ass: loc_lip_bound}, inequality~\eqref{eq: trajectories_close_under_a_t} and the condition of event \(B_t^1\) in the second inequality. Then, we applied the condition of \(B_t^2\) in the third inequality and finally Assumption~\ref{ass: loc_lip_bound} in the last inequality. Inserting this back into the left-hand side of \eqref{eq: equivalent_derivative_vanishing}, we discover
  \begin{align*}
    e^{\lambda(T-t)} N^\alpha |K^\epsilon(\mathbf{X}_t^{N,\epsilon}) - \overline{K}_t^\epsilon(\mathbf{Y}_t^{N,\epsilon}) |_\infty
    &\le e^{\lambda(T-t)} N^\alpha C(\rho_0,T) ( |\mathbf{X}_t^{N, \epsilon} -\mathbf{Y}_t^{N,\epsilon}|_\infty +N^{-(\delta+\alpha)} ) \\
    &=  C(\rho_0,T) e^{\lambda(T-t)} (N^\alpha |\mathbf{X}_t^{N, \epsilon} -\mathbf{Y}_t^{N,\epsilon}|_\infty +N^{-\delta}).
  \end{align*}
  Choosing \(\lambda=  C(\rho_0,T)\) provides \eqref{eq: equivalent_derivative_vanishing} and concludes the proof.
\end{proof}

\begin{remark}
  The cut-off \(\alpha \in (0,1/2)\) was only used in Corollary~\ref{cor: law_of_large_numbers} to bound the set \(B_t^2\). Hence, one possibility on improving the cut-off is to optimize  Proposition~\ref{prop: law_of_large_numbers} in order to handle more general cut-off functions.
\end{remark}

From Theorem~\ref{theorem: main_theorem} it immediately follows that the marginals of \(X_t^N\) and \(Y_t^N\) converge in the Wasserstein metric, see e.g. \cite[Corollary~2.2]{canizares2017stochastic}. For the sake of completeness we include the statement below.

\begin{corollary}\textup{\cite[Corollary~2.2.]{canizares2017stochastic}}\label{cor: mean_field_limit_for_regularized_trajectories}
  Let the assumptions of Theorem~\ref{theorem: main_theorem} hold. Consider the probability density \(  \rho_t^{\otimes N,\epsilon} \) of \(\mathbf{Y}_t^{N,\epsilon}\) and \(\rho_t^{N,\epsilon} \) the probability density of \(\mathbf{X}_t^{N,\epsilon}\). Then, \(\rho_t^{N,k,\epsilon}\) converges weakly (in the sense of measures) to \( \rho_t^{\otimes k ,\epsilon}\) as \(N \to \infty\), \(\epsilon(N) \to 0\) for each fixed \(k \ge 1 \). Furthermore, the probability density \(\rho_t^{N,\epsilon}\) converges weakly (in the sense of measures) to the same measure as \(  \rho_t^{\otimes N, \epsilon}\) as \(N \to \infty\). More precisely, there exists a positive constant \(C\) and \(N_0 \in \N\) such that
  \begin{equation*}
    \sup\limits_{t\in [0,T]} W_1( \rho_t^{N,k,\epsilon} , \rho_t^{\otimes k, \epsilon} ) ,  \,
    \sup\limits_{t\in [0,T]} W_1(\rho_t^{N,\epsilon} , \, \rho_t^{\otimes N, \epsilon})  \le C(\rho_0,T,\alpha) N^{-\alpha}
  \end{equation*}
  holds for each \(k \ge 1 \) and \(N \ge N_0\), where \(W_1\) denotes the Wasserstein metric
  \begin{equation*}
    W_1(\mu,\nu) := \inf\limits_{\pi \in \Pi(\mu,\nu)} \int_{ \R \times \R} \frac{1}{k} \sum\limits_{i=1}^k |x^{i}-y^{i}| \Id \pi (x,y)
  \end{equation*}
  and \(\Pi(\mu,\nu)\) is the set of all probability measures on \(\R \times \R\) with first marginal \(\mu \) and second marginal \(\nu\). The constant \(C(\rho_0,T,\alpha)\) depends on the initial condition \(\rho_0\), the final time \(T\) and \(\alpha\).  Moreover, \(N_0 \in \N\) is the same as in Theorem~\ref{theorem: main_theorem}.
\end{corollary}

Corollary~\ref{cor: mean_field_limit_for_regularized_trajectories} implies the weak convergence in the sense of measures of the \(k\)-th marginal \(\rho_t^{N,k,\epsilon} \) to the product measure \( \rho_t^{\otimes k} \). Indeed, since \(\rho_t^{N,k,\epsilon}\) converges weakly to \(\rho_t^{\otimes k, \epsilon} \), it is sufficient to show that \( \rho_t^{\otimes k , \epsilon}\) converges weakly to \( \rho_t^{\otimes k }\). By the classic result \cite[Proposition~2.2]{snitzman_propagation_of_chaos} we can consider the special case \(k=2\), i.e \(\rho_t^{\epsilon} \otimes \rho_t^{\epsilon} \) converges weakly to \( \rho_t  \otimes \rho_t\). We can further reduce it by applying \cite[Theorem~2.8]{billingsley2013convergence}, which tells us that it is enough to show \(\rho_t^\epsilon\) converges weakly to \(\rho_t\).

\begin{lemma}\label{lemma: weak_comvergence_regularized_mean_to_mean}
  Let \(T>0\) and suppose Assumption~\ref{ass: initial condition}. Moreover, let \((\rho^\epsilon, \epsilon > 0)\) and \(\rho\) be as in Corollary~\ref{cor: mean_field_limit_for_regularized_trajectories}. Then, one has
  \begin{equation}\label{eq: weak_comvergence_regularized_mean_to_mean}
    \sup_{t \in [0,T]} \left| \int_\R (\rho_t^\epsilon(x) -\rho_t(x)) \phi (x) \Id x  \right|
    \xrightarrow[ \epsilon \to 0]{} 0
  \end{equation}
  for all \(\phi \in L^\infty(\R)\). In particular, \(\rho_t^{\epsilon} \otimes \rho_t^{\epsilon} \) converges weakly to \( \rho_t  \otimes \rho_t\) for all \(t \ge 0\) in the sense of measures.
\end{lemma}

\begin{remark}
  Suppose the assumptions of Theorem~\ref{theorem: main_theorem} hold. The Lemma~\ref{lemma: weak_comvergence_regularized_mean_to_mean} together with the discussion before Lemma~\ref{lemma: uniform_bound_x2} and \cite[Proposition~2.2]{snitzman_propagation_of_chaos} imply that, for all \(t \in [0,T]\),
  \begin{equation*}
    \lim_{N \to \infty} \frac{1}{N} \sum\limits_{i =1}^N \delta_{X_t^{i,\epsilon}} = \rho_t
  \end{equation*}
  in law as measure valued random variables if \(\epsilon \sim N^{-\beta}\).
\end{remark}

\begin{proof}[Proof of Lemma~\ref{lemma: weak_comvergence_regularized_mean_to_mean}]
  First, we notice that the convergence is uniform in time. Therefore, the strong convergence result from Lemma~\ref{lemma: aubin_lion_space} cannot be applied. We start by showing \eqref{eq: weak_comvergence_regularized_mean_to_mean} holds for \(\phi \in H^1(\R)\). To that end, let us assume \(\phi\) is in a dense subset and smooth enough, i.e. \(\phi \in \testfunctions{\R}\).  Now, let \( 0 \le t_1 < t_2 \le T \). Then, the uniform bound on \(\frac{\d}{\d t} \rho^\epsilon_t\) (see \eqref{eq: regularized_solutions_estimate}) and integration by parts \cite[Theorem~23.23]{ZeidlerEberhard1990Nfaa} implies
  \begin{align*}
    \left| \int_\R \rho^\epsilon(t_1,x) \phi (x) \Id x  - \int_\R \rho^\epsilon(t_2,x) \phi (x) \Id x \right|
    &= \left| \int\limits_{t_1}^{t_2} \left \langle\frac{\d}{\d t} \rho^\epsilon_t, \phi \right \rangle_{H^{-1}(\R),H^1(\R)} \Id t  \right|  \\
    &\le |t_2-t_1|^{1/2} \norm{\frac{\d}{\d t} \rho^\epsilon}_{L^2([0,T];H^{-1}(\R))} \norm{\phi}_{H^{1}(\R)} \\
    &\le  C |t_2-t_1|^{1/2} \norm{\phi}_{H^{1}(\R)} .
  \end{align*}
  Consequently, the sequence of function \(t \mapsto \int_\R \rho_t^\epsilon(x) \phi(x) \Id x \) is equicontinuous. Using the \(L^\infty([0,T];L^2(\R))\)-bound, we also get a uniform bound on the sequence. As a result, we can apply the Arzela--Ascoli theorem to obtain a convergent subsequence, which depends on \(\phi\) and will be denoted by \((\rho^{\epsilon(\phi)}, \epsilon(\phi) \in \N)\) such that \(  \int_\R \rho_t^{\epsilon(\phi)} \phi \Id x \to \zeta(\phi)\) in \(C([0,T])\). By the fundamental lemma of calculus of variation and the fact that \( \rho_t^{\epsilon(\phi)} \) converges weakly in \(L^2([0,T];L^2(\R))\) we can identify the limit \(\zeta(\phi) = \int_\R \rho_t \phi \Id x \). Since \(\phi\) was taken from a dense subset of \( H^1(\R)\), we can use a diagonal argument to obtain a subsequence, which will be not renamed, such that, for \(\phi \in  H^1(\R)\),
  \begin{equation}\label{eq: weak_comvergence_regularized_mean_to_mean_in_h1}
    \sup_{t \in [0,T]} \left| \int_\R (\rho_t^{\epsilon(\phi)}(x) -\rho_t(x)) \phi (x) \Id x  \right|
    \xrightarrow[ \epsilon(\phi) \to \infty]{} 0 .
  \end{equation}
  With another density argument and the uniform bound of \((\rho^\epsilon, \epsilon \ge 0)\) in \(L^\infty([0,T];L^2(\R))\) we obtain for each \(\phi \in L^2(\R)\) a subsequence \((\rho^{\epsilon(\phi)}_t, \epsilon(\phi) \in \N)\) such that \eqref{eq: weak_comvergence_regularized_mean_to_mean_in_h1} holds. Again, since \(L^2(\R)\) is separable we can use another diagonal argument to show that we can obtain a subsequence \((\rho^{\epsilon_k}_t, k \in \N)\) such that \eqref{eq: weak_comvergence_regularized_mean_to_mean_in_h1} holds for all \(\phi \in L^2(\R)\). Notice that this subsequence is independent of the function \(\phi\). Furthermore, the uniqueness of the limit implies that \eqref{eq: weak_comvergence_regularized_mean_to_mean_in_h1} actually holds for any sequence \((\rho^{\epsilon(N)}_t, \epsilon(N) >0)\) itself, where \(\epsilon(N)\) is some sequence depending on \(N\) such that \(\epsilon(N) \to 0\) as \(N \to \infty\).

  Next, for \(\phi \in L^\infty(\R)\), we apply Lemma~\ref{lemma: uniform_bound_x2} and the fact that \(\phi(x) \indicator{\{ |x| \le R \}} \in L^2(\R)\) to find
  \begin{align*}
    &\sup_{t \in [0,T]} \left| \int_{\R} (\rho_t^{\epsilon}(x) -\rho_t(x)) \phi (x) \Id x  \right| \\
    &\quad\le \sup_{t \in [0,T]} \left| \int_{\R} (\rho_t^{\epsilon}(x) -\rho_t(x)) \phi (x) \indicator{\{ |x| \le R \}} \Id x  \right|
    +  \sup_{t \in [0,T]} \left| \int_{\R} (\rho_t^{\epsilon}(x) -\rho_t(x)) \phi (x) \indicator{\{ |x| \ge R \}} \Id x  \right|  \\
    &\quad\le \sup_{t \in [0,T]} \left| \int_{\R} (\rho_t^{\epsilon}(x) -\rho_t(x)) \phi (x) \indicator{\{ |x| \le R \}} \Id x  \right|
    + \norm{\phi}_{L^\infty(\R)} R^{-1} \sup_{t \in [0,T]}  \int_{\R} |\rho_t^{\epsilon}(x) +\rho_t(x)| |x| \Id x   \\
    &\quad\le  \sup_{t \in [0,T]} \left| \int_{\R} (\rho_t^{\epsilon}(x) -\rho_t(x)) \phi (x) \indicator{\{ |x| \le R \}} \Id x  \right|
    + C(\rho_0) \norm{\phi}_{L^\infty(\R)} R^{-1} .
  \end{align*}
  Letting \(\epsilon \to 0\) and then \(R \to \infty\), we obtain \eqref{eq: weak_comvergence_regularized_mean_to_mean} and the corollary is proven.
\end{proof}


\bibliography{quellen}

\def\cprime{$'$}
\providecommand{\bysame}{\leavevmode\hbox to3em{\hrulefill}\thinspace}
\providecommand{\MR}{\relax\ifhmode\unskip\space\fi MR }
\providecommand{\MRhref}[2]{%
  \href{http://www.ams.org/mathscinet-getitem?mr=#1}{#2}
}
\providecommand{\href}[2]{#2}
\begin{thebibliography}{CJLW17b}

\bibitem[AF03]{AdamsRobertA2003Ss}
Robert~A. Adams and John J.~F. Fournier, \emph{Sobolev spaces}, second ed.,
  Pure and Applied Mathematics (Amsterdam), vol. 140, Elsevier/Academic Press,
  Amsterdam, 2003.

\bibitem[BCC11]{BolleyCarrillo2011}
Fran\c{c}ois Bolley, Jos\'{e}~A. Ca{\~n}izo, and Jos\'{e}~A. Carrillo,
  \emph{Stochastic mean-field limit: non-{L}ipschitz forces and swarming},
  Math. Models Methods Appl. Sci. \textbf{21} (2011), no.~11, 2179--2210.

\bibitem[BDP06]{BlanchetPerthame2006}
Adrien Blanchet, Jean Dolbeault, and Beno\^{i}t Perthame, \emph{Two-dimensional
  {K}eller-{S}egel model: optimal critical mass and qualitative properties of
  the solutions}, Electron. J. Differential Equations (2006), No. 44, 32.

\bibitem[BJS22]{bresch2023}
Didier Bresch, Pierre-Emmanuel Jabin, and Juan Soler, \emph{A new approach to
  the mean-field limit of {V}lasov--{F}okker--{P}lanck equations}, ArXiv
  preprint arXiv:2203.15747 (2022).

\bibitem[BJW19]{BreschDidierJabinWangZhenfu2019}
Didier Bresch, Pierre-Emmanuel Jabin, and Zhenfu Wang, \emph{On mean-field
  limits and quantitative estimates with a large class of singular kernels:
  application to the {P}atlak-{K}eller-{S}egel model}, C. R. Math. Acad. Sci.
  Paris \textbf{357} (2019), no.~9, 708--720.

\bibitem[BJW20]{bresch2020}
\bysame, \emph{{Mean-field limit and quantitative estimates with singular
  attractive kernels}}, ArXiv preprint arXiv:2011.08022 (2020).

\bibitem[BR20]{barburoeckner2020}
Viorel Barbu and Michael R\"{o}ckner, \emph{From nonlinear {F}okker-{P}lanck
  equations to solutions of distribution dependent {SDE}}, Ann. Probab.
  \textbf{48} (2020), no.~4, 1902--1920.

\bibitem[Bre11]{BrezisHaim2011FaSs}
Haim Brezis, \emph{Functional analysis, {S}obolev spaces and partial
  differential equations}, Universitext, Springer, New York, 2011.

\bibitem[CCH14a]{CarilloChoiPil2014}
Jos\'{e}~Antonio Carrillo, Young-Pil Choi, and Maxime Hauray, \emph{The
  derivation of swarming models: mean-field limit and {W}asserstein distances},
  Collective dynamics from bacteria to crowds, CISM Courses and Lect., vol.
  553, Springer, Vienna, 2014, pp.~1--46.

\bibitem[CCH14b]{CarrilloChoi2014}
\bysame, \emph{The derivation of swarming models: mean-field limit and
  {W}asserstein distances}, Collective dynamics from bacteria to crowds, CISM
  Courses and Lect., vol. 553, Springer, Vienna, 2014, pp.~1--46.

\bibitem[CCS19]{CarrilloChoiSalem2019}
Jos\'{e}~A. Carrillo, Young-Pil Choi, and Samir Salem, \emph{Propagation of
  chaos for the {V}lasov-{P}oisson-{F}okker-{P}lanck equation with a polynomial
  cut-off}, Commun. Contemp. Math. \textbf{21} (2019), no.~4, 1850039, 28.

\bibitem[CG17]{canizares2017stochastic}
Ana Ca{\~n}izares~Garc\'{i}a, \emph{On a stochastic particle model of the
  {K}eller--{S}egel equation and its macroscopic limit}, Ph.D. thesis, Ludwig
  Maximilian University of Munich, 2017.

\bibitem[CJLW17a]{ChazelleJiu2017}
Bernard Chazelle, Quansen Jiu, Qianxiao Li, and Chu Wang, \emph{Well-posedness
  of the limiting equation of a noisy consensus model in opinion dynamics}, J.
  Differential Equations \textbf{263} (2017), no.~1, 365--397.

\bibitem[CJLW17b]{CHAZELLE2017365}
\bysame, \emph{Well-posedness of the limiting equation of a noisy consensus
  model in opinion dynamics}, J. Differential Equations \textbf{263} (2017),
  no.~1, 365--397.

\bibitem[CLPY20]{ChenLiPickl2020}
Li~Chen, Xin Li, Peter Pickl, and Qitao Yin, \emph{Combined mean field limit
  and non-relativistic limit of {V}lasov-{M}axwell particle system to
  {V}lasov-{P}oisson system}, J. Math. Phys. \textbf{61} (2020), no.~6, 061903,
  21.

\bibitem[DTGC14]{Domschke2014}
Pia Domschke, Dumitru Trucu, Alf Gerisch, and Mark A.~J. Chaplain,
  \emph{Mathematical modelling of cancer invasion: implications of cell
  adhesion variability for tumour infiltrative growth patterns}, J. Theoret.
  Biol. \textbf{361} (2014), 41--60.

\bibitem[Eva15]{EvansLawrenceC2015Pde}
Lawrence~C. Evans, \emph{Partial differential equations}, second edition,
  reprinted with corrections ed., Graduate studies in mathematics; volume 19,
  Providence, Rhode Island, 2015.

\bibitem[FHS19]{Fetecau_2019}
Razvan~C. Fetecau, Hui Huang, and Weiran Sun, \emph{Propagation of chaos for
  the {K}eller-{S}egel equation over bounded domains}, J. Differential
  Equations \textbf{266} (2019), no.~4, 2142--2174.

\bibitem[FJ17]{FournierJourdain2017}
Nicolas Fournier and Benjamin Jourdain, \emph{Stochastic particle approximation
  of the {K}eller-{S}egel equation and two-dimensional generalization of
  {B}essel processes}, Ann. Appl. Probab. \textbf{27} (2017), no.~5,
  2807--2861.

\bibitem[G{\"a}r88]{Gartner1988}
J{\"u}rgen G{\"a}rtner, \emph{On the {M}c{K}ean-{V}lasov limit for interacting
  diffusions}, Math. Nachr. \textbf{137} (1988), 197--248.

\bibitem[GQ15]{godinho_Keller_Segel2015}
David Godinho and Cristobal Quiñinao, \emph{{Propagation of chaos for a
  subcritical Keller–Segel model}}, Annales de l'Institut Henri Poincaré,
  Probabilités et Statistiques \textbf{51} (2015), no.~3, 965--992.

\bibitem[Han22]{han2023}
Yi~Han, \emph{Entropic propagation of chaos for mean field diffusion with $l^p$
  interactions via hierarchy, linear growth and fractional noise}, ArXiv
  preprint arXiv:2205.02772 (2022).

\bibitem[HK02]{hegselmann2002}
Rainer Hegselmann and Ulrich Krause, \emph{Opinion dynamics and bounded
  confidence models, analysis, and simulation}, Journal of {A}rtifical
  {S}ocieties and {S}ocial {S}imulation (JASSS) \textbf{5} (2002), no.~3.

\bibitem[HKPZ19]{HaKimPickl2019}
Seung-Yeal Ha, Jeongho Kim, Peter Pickl, and Xiongtao Zhang, \emph{A
  probabilistic approach for the mean-field limit to the {C}ucker-{S}male model
  with a singular communication}, Kinet. Relat. Models \textbf{12} (2019),
  no.~5, 1045--1067.

\bibitem[HL09]{HaLiu2009}
Seung-Yeal Ha and Jian-Guo Liu, \emph{A simple proof of the {C}ucker-{S}male
  flocking dynamics and mean-field limit}, Commun. Math. Sci. \textbf{7}
  (2009), no.~2, 297--325.

\bibitem[HLL19]{HuangHiuLius2019}
Hui Huang, Jian-Guo Liu, and Jianfeng Lu, \emph{Learning interacting particle
  systems: diffusion parameter estimation for aggregation equations}, Math.
  Models Methods Appl. Sci. \textbf{29} (2019), no.~1, 1--29.

\bibitem[HLP20]{HuangLiuPickl2020}
Hui Huang, Jian-Guo Liu, and Peter Pickl, \emph{On the mean-field limit for the
  {V}lasov-{P}oisson-{F}okker-{P}lanck system}, J. Stat. Phys. \textbf{181}
  (2020), no.~5, 1915--1965.

\bibitem[HM14]{HaurayMischler2014}
Maxime Hauray and St\'{e}phane Mischler, \emph{On {K}ac's chaos and related
  problems}, J. Funct. Anal. \textbf{266} (2014), no.~10, 6055--6157.

\bibitem[Hor04]{Horstmann2004}
Dirk Horstmann, \emph{From 1970 until present: the {K}eller-{S}egel model in
  chemotaxis and its consequences. {II}}, Jahresber. Deutsch. Math.-Verein.
  \textbf{106} (2004), no.~2, 51--69.

\bibitem[Hos20]{noorazar2020recent}
Noorazar Hossein, \emph{Recent advances in opinion propagation dynamics: a 2020
  survey}, The European Physical Journal Plus \textbf{135} (2020), no.~6,
  1--20.

\bibitem[HP09]{HillenPainter2009}
T.~Hillen and K.~J. Painter, \emph{A user's guide to {PDE} models for
  chemotaxis}, J. Math. Biol. \textbf{58} (2009), no.~1-2, 183--217.

\bibitem[HRW21]{HuangPanpanWang2021}
Xing Huang, Panpan Ren, and Feng-Yu Wang, \emph{Distribution dependent
  stochastic differential equations}, Front. Math. China \textbf{16} (2021),
  no.~2, 257--301.

\bibitem[HRZ22]{hao2022}
Zimo Hao, Michael R{\"o}ckner, and Xicheng Zhang, \emph{Strong convergence of
  propagation of chaos for {M}c{K}ean--{V}lasov {SDE}s with singular
  interactions}, ArXiv preprint arXiv:2204.07952 (2022).

\bibitem[HT06]{Hagooddini2006}
John~W. Hagood and Brian~S. Thomson, \emph{Recovering a {F}unction from a
  {D}ini {D}erivative}, The American Mathematical Monthly \textbf{113} (2006),
  no.~1, 34--46.

\bibitem[Jab14]{JabinEmanuel2014}
Pierre-Emmanuel Jabin, \emph{A review of the mean field limits for {V}lasov
  equations}, Kinet. Relat. Models \textbf{7} (2014), no.~4, 661--711.

\bibitem[JW16]{JabinWangZhenfu2016}
Pierre-Emmanuel Jabin and Zhenfu Wang, \emph{Mean field limit and propagation
  of chaos for {V}lasov systems with bounded forces}, J. Funct. Anal.
  \textbf{271} (2016), no.~12, 3588--3627.

\bibitem[JW18]{JabinWang2018}
\bysame, \emph{Quantitative estimates of propagation of chaos for stochastic
  systems with {$W^{-1,\infty}$} kernels}, Invent. Math. \textbf{214} (2018),
  no.~1, 523--591.

\bibitem[KS70]{KellerSegel1970}
Evelyn~F. Keller and Lee~A. Segel, \emph{Initiation of slime mold aggregation
  viewed as an instability}, J. Theoret. Biol. \textbf{26} (1970), no.~3,
  399--415.

\bibitem[KS91]{KaratzasIoannis2009Bmas}
Ioannis Karatzas and Steven~E. Shreve, \emph{Brownian motion and stochastic
  calculus}, second ed., Graduate Texts in Mathematics, vol. 113,
  Springer-Verlag, New York, 1991.

\bibitem[Lac23]{Lacker2023}
Daniel Lacker, \emph{Hierarchies, entropy, and quantitative propagation of
  chaos for mean field diffusions}, Probab. Math. Phys. \textbf{4} (2023),
  no.~2, 377--432.

\bibitem[Leo17]{LeoniGiovanni2017Afci}
Giovanni Leoni, \emph{A first course in {S}obolev spaces}, second ed., Graduate
  Studies in Mathematics, vol. 181, American Mathematical Society, Providence,
  RI, 2017.

\bibitem[LL01]{LiebElliottH2010A}
Elliott~H. Lieb and Michael Loss, \emph{Analysis}, second ed., Graduate Studies
  in Mathematics, vol.~14, American Mathematical Society, Providence, RI, 2001.

\bibitem[LLY19]{LiLeiLiu2019}
Lei Li, Jian-Guo Liu, and Pu~Yu, \emph{On the mean field limit for {B}rownian
  particles with {C}oulomb interaction in 3{D}}, J. Math. Phys. \textbf{60}
  (2019), no.~11, 111501, 34.

\bibitem[Lor07a]{LORENZ_2007_bounded_confidence_survey}
Jan Lorenz, \emph{{C}ontinuous {O}pinion {D}ynamics {U}nder {B}ounded
  {C}onfidence: A {S}urvey}, International Journal of Modern Physics C
  \textbf{18} (2007), no.~12, 1819--1838.

\bibitem[Lor07b]{Lorenz2007}
\bysame, \emph{{C}ontinuous {O}pinion {D}ynamics {U}nder {B}ounded
  {C}onfidence: A {S}urvey}, International Journal of Modern Physics C
  \textbf{18} (2007), no.~12, 1819--1838.

\bibitem[LP91]{LionsPerhame1991}
P.-L. Lions and B.~Perthame, \emph{Propagation of moments and regularity for
  the {$3$}-dimensional {V}lasov-{P}oisson system}, Invent. Math. \textbf{105}
  (1991), no.~2, 415--430.

\bibitem[LP17]{lazarovici2017mean}
Dustin Lazarovici and Peter Pickl, \emph{A mean field limit for the
  {V}lasov--{P}oisson system}, Arch. Ration. Mech. Anal. \textbf{225} (2017),
  no.~3, 1201--1231.

\bibitem[LY19]{LiuYang2019}
Jian-Guo Liu and Rong Yang, \emph{Propagation of chaos for the {K}eller-{S}egel
  equation with a logarithmic cut-off}, Methods Appl. Anal. \textbf{26} (2019),
  no.~4, 319--347.

\bibitem[McK67]{McKean1967}
H.~P. McKean, Jr., \emph{Propagation of chaos for a class of non-linear
  parabolic equations}, Stochastic {D}ifferential {E}quations ({L}ecture
  {S}eries in {D}ifferential {E}quations, {S}ession 7, {C}atholic {U}niv.,
  1967), Air Force Office Sci. Res., Arlington, Va., 1967, pp.~41--57.

\bibitem[MV20]{mishura2020}
Yuliya Mishura and Alexander Veretennikov, \emph{Existence and uniqueness
  theorems for solutions of {M}c{K}ean-{V}lasov stochastic equations}, Theory
  Probab. Math. Statist. (2020), no.~103, 59--101.

\bibitem[Nir59]{ASNSP_1959_3_13_2_115_0}
Louis Nirenberg, \emph{On elliptic partial differential equations}, Annali
  della Scuola Normale Superiore di Pisa - Classe di Scienze \textbf{Ser. 3,
  13} (1959), no.~2, 115--162.

\bibitem[NRS22]{Nguyen2022}
Quoc-Hung Nguyen, Matthew Rosenzweig, and Sylvia Serfaty, \emph{Mean-field
  limits of {R}iesz-type singular flows}, Ars Inven. Anal. (2022), Paper No. 4,
  45.

\bibitem[Oel84]{Oeschlager1984}
Karl Oelschl\"{a}ger, \emph{A martingale approach to the law of large numbers
  for weakly interacting stochastic processes}, Ann. Probab. \textbf{12}
  (1984), no.~2, 458--479.

\bibitem[Osa87]{Hirofumi1987}
Hirofumi Osada, \emph{Propagation of chaos for the two-dimensional
  {N}avier-{S}tokes equation}, Probabilistic methods in mathematical physics
  ({K}atata/{K}yoto, 1985), Academic Press, Boston, MA, 1987, pp.~303--334.

\bibitem[Pat13]{billingsley2013convergence}
Billingsley Patrick, \emph{Convergence of probability measures}, John Wiley \&
  Sons, 2013.

\bibitem[Rom18]{romito2018}
Marco Romito, \emph{A simple method for the existence of a density for
  stochastic evolutions with rough coefficients}, Electron. J. Probab.
  \textbf{23} (2018), Paper no. 113, 43.

\bibitem[RS23]{RosenzweigSerfaty2023}
Matthew Rosenzweig and Sylvia Serfaty, \emph{Global-in-time mean-field
  convergence for singular {R}iesz-type diffusive flows}, Ann. Appl. Probab.
  \textbf{33} (2023), no.~2, 754--798.

\bibitem[RZ21]{rockner_DDSDE2021}
Michael R\"{o}ckner and Xicheng Zhang, \emph{Well-posedness of distribution
  dependent {SDE}s with singular drifts}, Bernoulli \textbf{27} (2021), no.~2,
  1131--1158.

\bibitem[Ser20]{Serfaty_2020}
Sylvia Serfaty, \emph{Mean field limit for {C}oulomb-type flows}, Duke Math. J.
  \textbf{169} (2020), no.~15, 2887--2935.

\bibitem[Sho97]{showalter1997}
R.~E. Showalter, \emph{Monotone operators in {B}anach space and nonlinear
  partial differential equations}, Mathematical Surveys and Monographs,
  vol.~49, American Mathematical Society, Providence, RI, 1997.

\bibitem[Szn91]{snitzman_propagation_of_chaos}
Alain-Sol Sznitman, \emph{Topics in propagation of chaos}, {\'E}cole
  d'{\'E}t{\'e} de {P}robabilit{\'e}s de {S}aint-{F}lour {XIX}---1989, Lecture
  Notes in Math., vol. 1464, Springer, Berlin, 1991, pp.~165--251.

\bibitem[TBL06]{Topaz2006}
Chad~M. Topaz, Andrea~L. Bertozzi, and Mark~A. Lewis, \emph{A nonlocal
  continuum model for biological aggregation}, Bull. Math. Biol. \textbf{68}
  (2006), no.~7, 1601--1623.

\bibitem[Wan18]{Wang2018}
Feng-Yu Wang, \emph{Distribution dependent {SDE}s for {L}andau type equations},
  Stochastic Process. Appl. \textbf{128} (2018), no.~2, 595--621.

\bibitem[Wan23]{Wang2023}
\bysame, \emph{A new type distribution-dependent {SDE} for singular nonlinear
  {PDE}}, J. Evol. Equ. \textbf{23} (2023), no.~2, Paper No. 35, 30.

\bibitem[Yos80]{YoshidaKosaku1995FA}
K\^{o}saku Yosida, \emph{Functional analysis}, sixth ed., Grundlehren der
  Mathematischen Wissenschaften [Fundamental Principles of Mathematical
  Sciences], vol. 123, Springer-Verlag, Berlin-New York, 1980.

\bibitem[Zei90]{ZeidlerEberhard1990Nfaa}
Eberhard Zeidler, \emph{Nonlinear functional analysis and its applications.
  2{A}: {L}inear monotone operators}, New York Berlin Heidelberg, 1990.

\end{thebibliography}
\bibliographystyle{amsalpha}

\end{document}